\documentclass[a4paper,oneside, reqno]{amsart}
\usepackage{mathrsfs}
\usepackage{amsfonts}
\usepackage{amssymb}
\usepackage{amsxtra}
\usepackage{dsfont}

\usepackage[english,polish]{babel}

\newcommand{\pl}[1]{\foreignlanguage{polish}{#1}}

\usepackage{color}
\usepackage{enumitem}

\theoremstyle{plain}
\newtheorem{theorem}{Theorem}[section]

\newtheorem{proposition}{Proposition}[section]

\newtheorem{lemma}[proposition]{Lemma}
\theoremstyle{definition}

\numberwithin{equation}{section}

\newcounter{thm}

\theoremstyle{plain}

\newcommand{\RR}{\mathbb{R}}
\newcommand{\ZZ}{\mathbb{Z}}
\newcommand{\TT}{\mathbb{T}}
\newcommand{\CC}{\mathbb{C}}
\newcommand{\NN}{\mathbb{N}}

\newcommand{\calF}{\mathcal{F}}

\newcommand{\ind}[1]{{\mathds{1}_{{#1}}}}
\newcommand{\dist}{\operatorname{dist}}

\newcommand{\sprod}[2] {{#1 \cdot #2}}

\newcommand{\dif}{\mathrm{d}}

\DeclareMathOperator{\supp}{supp}

\title[On dimension-free estimates in  $\ZZ^d$]
{On discrete Hardy--Littlewood maximal functions \\
  over the balls in
  $\mathbb Z^d$: dimension-free estimates }

\author{Jean Bourgain}
\address{Jean Bourgain \\
  School of Mathematics\\
  Institute for Advanced Study\\
  Princeton, NJ 08540\\
  USA}
\email{bourgain@math.ias.edu}

\author{Mariusz Mirek}

\address{Mariusz Mirek \\
  Department of Mathematics\\
  Rutgers University\\
Piscataway, NJ 08854\\ USA \&
	Instytut Matematyczny\\
	Uniwersytet \pl{Wroc{\lll}awski}\\
	Plac Grun\-waldzki 2/4\\
	50-384 \pl{Wroc{\lll}aw}\\
	Poland}
\email{mariusz.mirek@rutgers.edu}

\author{Elias M. Stein}
\address{
	Elias M. Stein\\
	Department of Mathematics\\
	Princeton University\\
	Princeton\\
	NJ 08544-100 USA}
\email{stein@math.princeton.edu}

\author{B{\l}a{\.z}ej Wr{\'o}bel}
\address{ B{\l}a{\.z}ej Wr{\'o}bel\\
	Instytut Matematyczny\\
	Uniwersytet \pl{Wroc{\lll}awski}\\
	Plac Grun\-waldzki 2/4\\
	50-384 \pl{Wroc{\lll}aw}\\
	Poland}
\email{blazej.wrobel@math.uni.wroc.pl}

\thanks{ Jean Bourgain was supported by NSF grant DMS-1800640.
Mariusz Mirek was partially supported by the Schmidt Fellowship and
the IAS School of Mathematics and by the National Science Center, Poland grant
DEC-2015/19/B/ST1/01149.  Elias M. Stein was partially supported by
NSF grant DMS-1265524.  B{\l}a{\.z}ej Wr{\'o}bel was partially
supported by the National Science Centre, Poland grant Opus
2018/31/B/ST1/00204}

\begin{document}
 
\selectlanguage{english}

\begin{abstract}
We show that the discrete  Hardy--Littlewood maximal functions associated with the Euclidean balls in $\mathbb Z^d$ with dyadic radii  have  bounds independent of the dimension on  $\ell^p(\mathbb Z^d)$ for $p\in[2, \infty]$.
\end{abstract}

\maketitle

\section{Introduction}
\label{sec:1}
\subsection{Motivations and statement of the results} Let $G$  be a convex centrally symmetric body in $\RR^d$, which is simply a bounded closed and centrally symmetric convex subset of $\RR^d$ with non-empty interior. An important class of convex symmetric bodies in $\RR^d$ are $q$-balls 
\begin{align}
\label{eq:88}
\begin{split}
  B^q=\Big\{x\in\RR^d\colon& |x|_q=\Big(\sum_{1\le k\le
  d}|x_k|^q\Big)^{1/q}\le 1\Big\} \quad \text{for} \quad q\in[1, \infty),\\
B^{\infty}&=\{x\in\RR^d\colon|x|_{\infty}=\max_{1\le k\le d}|x_k|\le 1\}.     
\end{split}
\end{align}
For every $t>0$ and  for every  $x\in\RR^d$ we define the integral Hardy--Littlewood averaging operator
\begin{align}
\label{eq:93}
M_t^Gf(x)=\frac{1}{|G_t|}\int_{G_t}f(x-y){\rm d}y \quad \text{for} \quad f\in L^1_{\rm loc}(\RR^d),
\end{align}
where $G_t=\{y\in\RR^d: t^{-1}y\in G\}$. For $p\in(1, \infty]$, let $C_p(d,G)>0$ be the best constant
such that   the following inequality 
\begin{align}
\label{eq:109}
   \big\|\sup_{t>0}|M_t^Gf|\big\|_{L^p(\RR^d)}\le C_p(d,G)\|f\|_{L^p(\RR^d)}
\end{align}
holds for every $f\in L^p(\RR^d)$. If $p=\infty$, then \eqref{eq:109} holds with $C_p(d,G)=1$, since $M_t^G$ is the averaging operator. By appealing to the real interpolation and  a covering argument for $p=1$, it is not difficult to see that $C_p(d,G)<\infty$ for every $p\in(1, \infty)$ and for every convex symmetric body $G\subset\RR^d$.

In  the case of the Euclidean balls $G=B^2$ the theory of spherical maximal functions was used  \cite{SteinMax}  to show that $C_p(d,B^2)$ is bounded independently of the dimension for every $p\in(1, \infty]$. 
Not long afterwards it was shown, in \cite{B1} for $p=2$, and in \cite{B2, Car1} for $p\in(3/2, \infty]$, that $C_p(d,G)$ is bounded by an absolute constant, which is independent of the underlying convex symmetric  body $G\subset\RR^d$. 
However, if the supremum in \eqref{eq:109} is taken over a dyadic set, i.e. $t\in\mathbb D=\{2^n:n\in\NN\cup\{0\}\}$, then \eqref{eq:109} holds for all $p\in(1, \infty]$ and  $C_p(d,G)$ is independent of the body $G\subset\RR^d$ as well.

It is conjectured that the inequality in \eqref{eq:109} holds for all
$p\in(1, \infty]$ and for all convex symmetric bodies $G\subset\RR^d$
with $C_p(d,G)$ independent of $d$. It is reasonable to believe that
this is true, since it was verified for a large class of convex
symmetric bodies. Namely, for the $q$-balls $G=B^q$ the full range
$p\in(1, \infty]$ of dimension-free estimates for $C_p(d,B^q)$ was
established in \cite{Mul1} (for $q\in [1, \infty)$) and in \cite{B3}
(for cubes $q=\infty$) with constants depending only on $q$. The
general case is beyond our reach at this point.  We refer also to the
survey article \cite{DGM1} for a very careful and exhaustive
exposition of the subject, and see also \cite{BMSW1} and \cite{MSZ0,MSZ1}
for extensions of dimension-free estimates to $r$-variational and jump
inequalities.

However, similar questions have been recently
investigated by the authors \cite{BMSW3} for the discrete analogues of
the operators $M_t^G$ in $\ZZ^d$. The aim of the present article is to
continue the investigations in this direction.

 For every $t>0$ and  for every $x\in\ZZ^d$ we define  the discrete Hardy--Littlewood averaging operator 
\begin{align}
\label{eq:85}
\mathcal M_t^Gf(x)=\frac{1}{|G_t\cap \ZZ^d|}\sum_{y\in G_t\cap\ZZ^d}f(x-y) \quad \text{for} \quad f\in\ell^1(\ZZ^d).
\end{align}  
We note that the operator $\mathcal M_t^G$ is a discrete analogue of \eqref{eq:93}.

For $p\in(1, \infty]$, let $\mathcal C_p(d,G)>0$ be the best constant
such that the following  inequality 
\begin{align}
\label{eq:86}
   \big\|\sup_{t>0}|\mathcal M_t^Gf|\big\|_{\ell^p(\ZZ^d)}\le \mathcal C_p(d,G)\|f\|_{\ell^p(\ZZ^d)}
\end{align}
holds for every $f\in\ell^p(\ZZ^d)$. Arguing in a similar way as in \eqref{eq:109} we conclude  that $\mathcal C_p(d,G)<\infty$ for every $p\in(1, \infty]$ and for every convex symmetric body $G\subset\RR^d$. 

The question now is to decide whether $\mathcal C_p(d,G)$ can be bounded independently of the dimension $d$ for every $p\in(1, \infty)$. 
In \cite{BMSW3} the authors examined this question in the case of the discrete cubes $B^{\infty}\cap\ZZ^d$, and showed that for every $p\in(3/2, \infty]$ there is a constant $C_p>0$ independent of the dimension such that $\mathcal C_p(d,B^{\infty})\le C_p$. It was also shown in \cite{BMSW3} that if the supremum in \eqref{eq:86} is restricted to the dyadic set $\mathbb D$, then \eqref{eq:86} holds for all $p\in(1, \infty]$ and  $\mathcal C_p(d,G)$ is independent of the dimension.

On the other hand, we constructed in \cite{BMSW3}  a simple example of a  convex symmetric body in $\ZZ^d$ for which maximal estimate \eqref{eq:86} on $\ell^p(\ZZ^d)$ involves the smallest constant $\mathcal C_{p}(d, G)>0$ unbounded in
$d$ for every $p\in(1, \infty)$. In order to carry out the construction it suffices to  fix a sequence  $1\leq \lambda_1<\ldots<\lambda_d<\ldots<\sqrt{2}$ and consider  the ellipsoid 
\begin{align*}
E(d)=\Big\{x\in \RR^d\colon \sum_{k=1}^d \lambda_k^2x_k^2\,\le 1 \Big\}.
\end{align*}
Then one can prove that for every $p\in(1, \infty)$ there is $C_p>0$ such that for every $d\in\NN$ one has
\begin{align}
\label{eq:91}
\mathcal C_{p}(d, E(d))\ge C_p(\log d)^{1/p}.
\end{align}

Inequality \eqref{eq:91} shows that the dimension-free phenomenon  for the
Hardy--Littlewood maximal functions in the discrete setting is much more delicate and  not as broad as in the continuous case. All these results give us strong motivation to understand the situation more generally, in particular  in the case of $q$-balls $G=B^q$ where $q\in[1, \infty)$, see \eqref{eq:88}, which is well understood in the continuous setup. 

The main purpose of this work is to prove a dyadic variant of inequality \eqref{eq:86} for \eqref{eq:85} with $G=B^2$.
\begin{theorem}
\label{thm:0}
For every $p\in[2, \infty]$ there exists a constant $C_p>0$ independent of $d\in\NN$ such that for every 
$f\in\ell^p(\ZZ^d)$  we have
\begin{align}
\label{eq:90}
   \big\|\sup_{N\in\mathbb D}|\mathcal M_N^{B^2}f|\big\|_{\ell^p(\ZZ^d)}\le C_p\|f\|_{\ell^p(\ZZ^d)}.
\end{align}
\end{theorem}

We shall briefly outline the strategy for proving Theorem \ref{thm:0}. By a simple interpolation the proof of inequality \eqref{eq:90} is only interesting for $p = 2$, and it will consist of three steps. In the consecutive steps, we shall consider maximal functions corresponding to the operators $\mathcal M_N^{B^2}$ in which the supremum is restricted respectively to the sets:
\begin{enumerate}[label*={\arabic*}.]
\item $\mathbb D_{C_3, \infty}=\{N\in\mathbb D:N\ge C_3d\}$, the large-scale case;
\item $\mathbb D_{C_1, C_2}=\{N\in\mathbb D:C_1d^{1/2}\le N\le C_2d\}$, the intermediate-scale case;
\item $\mathbb D_{C_0}=\{N\in\mathbb D:N\le C_0d^{1/2}\}$, the small-scale case;
\end{enumerate}
for some universal constants $C_0, C_1, C_2, C_3>0$. Since we are working with the dyadic numbers $\mathbb D$ the exact values of $C_0, C_1, C_2, C_3$ will never play a role as long as they are absolute constants. Moreover, the implied constants will be always allowed to depend on $C_0, C_1, C_2, C_3$. 

\subsection{The large-scale case} In this step, we will appeal to the comparison principle from \cite{BMSW4}, where it was shown that there are absolute  constants $C, c>0$ such that for every $p\in(1, \infty)$ and for every $f\in\ell^p(\ZZ^d)$  we have
\begin{align}
\label{eq:117}
\big\|\sup_{t\ge cd}|\mathcal M_t^{B^2}f|\big\|_{\ell^p(\ZZ^d)}\le CC_p(d,B^2)\|f\|_{\ell^p(\ZZ^d)}.
\end{align}
Inequality \eqref{eq:117} combined with the dimension-free estimates for $C_p(d,B^2)$ from \eqref{eq:109} yield the estimates for the full maximal function in the large-scale case. This is the easiest case, the remaining two cases, where it will be important that we are working with the dyadic numbers $\mathbb D$, are much more challenging.

\subsection{The intermediate-scale case} This case will be discussed in Section \ref{sec:3}, where we shall bound the maximal function corresponding to $\mathcal M_N^{B^2}$ with the supremum taken  over the set $\mathbb D_{C_1, C_2}$, see Theorem \ref{thm:1}. In this case, by a comparison $\mathcal M_N^{B^2}$ with a suitable semigroup $P_t$ (see \eqref{eq:119}, and also \cite{BMSW3}), the proof will be reduced, using a standard square function argument, to estimates of the multipliers $\mathfrak m_N^{B^2}$ associated with the operators $\mathcal M_N^{B^2}$.

The main objective of Section \ref{sec:3} is to show that there is a constant $C>0$ independent of $d\in\NN$ such that for every $\xi\in\TT^d$ we have 
\begin{align}
\label{eq:118}
  |\mathfrak m_N^{B^2}(\xi)-1|&\le C(\kappa(d, N)\|\xi\|)^2,\\
  \label{eq:120}|
  \mathfrak m_N^{B^2}(\xi)|&\le C\big((\kappa(d, N)\|\xi\|)^{-1}+\kappa(d, N)^{-\frac{1}{7}}\big),
\end{align}
where $\|\xi\|^2=\sum_{i=1}^d\|\xi_i\|^2$, and $\|\xi_i\|=\dist(\xi_i, \ZZ)$ for all $i\in\NN_d=\{1,2,\ldots,d\}$, and  $\kappa(d, N)=Nd^{-1/2}$ is the proportionality factor, which can be identified with the isotropic constant corresponding to the Euclidean ball $B_N^2$ with radius $N>0$, see \eqref{eq:25} and Lemma \ref{lem:19}. We also refer to \cite{B1} for more details.

The proof of inequality \eqref{eq:118} is given in Proposition \ref{prop:0} and relies on the invariance of $B_N^2\cap\ZZ^d$ under the permutation group of $\NN_d$. These invariance properties of $B_N^2\cap\ZZ^d$ play important roles in the whole article and allow us to exploit probabilistic arguments on the permutation groups.

The proof of inequality \eqref{eq:120} is given in Proposition \ref{prop:2} and it requires a more sophisticated analysis, and in particular three tools that we now highlight:
\begin{enumerate}
\item[(i)] Lemma \ref{lem:5}, which tells us, to some extent, that a significant amount of mass of $B_N^2\cap\ZZ^d$, like in the continuous setup, is concentrated near the boundary of $B_N^2\cap\ZZ^d$. This lemma combined with Lemma \ref{lem:6}, which is a variant of a concentration inequality for the hypergeometric distribution, leads us to a decrease dimension trick described in Lemma \ref{lem:8}.   
\item[(ii)] Lemma \ref{lem:9},  which is an outgrowth of the idea implicit in Lemma \ref{lem:8},  permits us to control the  multiplier $\mathfrak m_N^{B^2}$ by multipliers corresponding to the averages associated with balls in lower dimensional spaces, and consequently exploit the estimates for the multipliers corresponding to the operators $M_t^{B^2}$ in the continuous setting, see Theorem \ref{thm:3} and Lemma \ref{lem:20}. 
\item[(iii)] A convexity lemma described in Lemma \ref{lem:7}, which  is essential in the proof of inequality \eqref{eq:120}.  
\end{enumerate}
Let us remark that if we could prove the inequality
\begin{align}
\label{eq:121}
  |\mathfrak m_N^{B^2}(\xi)|&\le C(\kappa(d, N)\|\xi\|)^{-1},
\end{align}
instead of \eqref{eq:120}, then we would be able to extend  inequality \eqref{eq:90} with $N$ restricted to the set $\mathbb D_{C_1, C_2}$ for all $p\in(1, \infty]$. However, this will surely require new methods. 

\subsection{The small-scale case} This case will be discussed in Section \ref{sec:4}, where we shall be bounding the maximal function corresponding to $\mathcal M_N^{B^2}$ with the supremum taken  over the set $\mathbb D_{C_0}$, see Theorem \ref{thm:2}. Our strategy will be much the same as for the proof in the previous case. We shall find suitable approximating multipliers and reduce the matters to the square function estimates using Proposition \ref{prop:4}. However, this  case will require a more sophisticated analysis, due to its different  nature that becomes apparent in Lemma \ref{lem:15}, which says, to a certain degree, that a large percentage of mass of $B_N^2\cap\ZZ^d$ is concentrated on the set $\{-1,0, 1\}^d$. This observation allows us to employ the properties of  the Krawtchouk polynomials \eqref{eq:68}, as in \cite{HKS},  to prove Proposition \ref{prop:5}, which is the core of the proof of Proposition \ref{prop:4}. Using a uniform bound for the Krawtchouk polynomials (see Property \ref{item:5} in Theorem \ref{thm:100}) we are able to deduce a decay of the multipliers $\mathfrak m_N^{B^2}$ at infinity. Namely, we show (see Proposition \ref{prop:5}) that there are absolute constants $C, c>0$ such that for every $\xi\in\TT^d$ we have      
\begin{align}
\label{eq:123}
|\mathfrak m_N^{B^2}(\xi)|\le Ce^{-\frac{c\kappa(d, N)^2}{100}\sum_{i=1}^d\sin^2(\pi\xi_i)}
+Ce^{-\frac{c\kappa(d, N)^2}{100}\sum_{i=1}^d\cos^2(\pi\xi_i)}.
\end{align}
As it was proven in Proposition \ref{prop:4}, inequality \eqref{eq:123}, while different from \eqref{eq:120} or \eqref{eq:121}, is good enough to provide $\ell^2(\ZZ^d)$ theory  for the maximal function associated with $\mathcal M_N^{B^2}$ in which the supremum is restricted the set $\mathbb D_{C_0}$.  However, it is not clear at this moment whether \eqref{eq:123} can be used to give an extension of \eqref{eq:90} for some $p\in(1, 2)$ in the small-scale case. 

Finally some comments are in order. Currently our methods are limited to $\ell^2(\ZZ^d)$ theory for the dyadic maximal function $\sup_{N\in\mathbb D}|\mathcal M_N^{B^2}f|$. It is clear that more information must be provided, if one thinks about an extension of \eqref{eq:90} to the full maximal inequality as in  \eqref{eq:86} with $G=B^2$, even for $p=2$.  If we knew that \eqref{eq:121} holds and additionally we could control the difference of $\mathfrak m_N^{B^2}$, let us say, in the following sense: that is for every $N\in\NN$ and $\xi\in\TT^d$ we would have
\begin{align}
\label{eq:1}
|\mathfrak m_{N+1}^{B^2}(\xi)-\mathfrak m_N^{B^2}(\xi)|\le CN^{-1}
\end{align}
for some constant independent of the dimension $d\in\NN$. Then using the methods from \cite{BMSW3} or \cite{MSZ1}, and taking into account  \eqref{eq:118}, \eqref{eq:121} and \eqref{eq:1} we would obtain  that for every $p\in(3/2, \infty]$ there is a constant $C_p>0$ independent of the dimension such that $\mathcal C_p(d,B^{2})\le C_p$. We hope to
return to these questions in the near future.

\subsection{Notation}
\begin{enumerate}[label*={\arabic*}.]

\item From now on we shall use abbreviated notation and we will write $B_t=B_t^2$, and
$Q_t=B^{\infty}_t$, and $Q=Q_{1/2}=[-1/2, 1/2]^d$,  and also $|x|=|x|_2$ for any $x\in\RR^d$. Moreover,
$M_t=M_t^{B^2}$, and $\mathcal M_t=\mathcal M_t^{B^2}$ for any $t>0$.

\item Throughout the whole paper $d\in\NN$ will denote the dimension and
$C,c, C_0, C_1, \ldots>0$ will be absolute constants which do not depend on the
dimension, however their values may vary from line to line.
We will use the convention that $A \lesssim_{\delta} B$
($A \gtrsim_{\delta} B$) to say that there is an absolute constant
$C_{\delta}>0$ (which possibly depends on $\delta>0$) such that
$A\le C_{\delta}B$ ($A\ge C_{\delta}B$).  We will write
$A \simeq_{\delta} B$ when $A \lesssim_{\delta} B$ and
$A\gtrsim_{\delta} B$ hold simultaneously.

\item 
Let $\NN=\{1,2,\ldots\}$ be the set of positive integers and $\NN_0 = \NN\cup\{0\}$, and
$\mathbb D=\{2^n: n\in\NN_0\}$ will denote the set of all dyadic numbers.
We set $\NN_N = \{1, 2, \ldots, N\}$ for any $N \in \NN$.

\item 
The Euclidean space $\RR^d$
is endowed with the standard inner product
\[
x\cdot\xi=\langle x, \xi\rangle=\sum_{k=1}^dx_k\xi_k
  \]
for every $x=(x_1,\ldots, x_d)$ and $\xi=(\xi_1, \ldots,
\xi_d)\in\RR^d$.

\item For a countable set $\mathcal Z$ endowed with the counting measure we will
write for any $p\in[1, \infty]$ that
\[
\ell^p(\mathcal Z)=\{f:\mathcal Z\to \CC: \|f\|_{\ell^p(\mathcal Z)}<\infty\},
\]
where for any $p\in[1, \infty)$ we have
\begin{align*}                    
  \|f\|_{\ell^p(\mathcal Z)}=\Big(\sum_{m\in\mathcal Z}|f(m)|^p\Big)^{1/p} \qquad \text{and} \qquad
  \|f\|_{\ell^{\infty}(\mathcal Z)}=\sup_{m\in\mathcal Z}|f(m)|.
\end{align*}
In our case usually $\mathcal Z=\ZZ^d$. We will also abbreviate $\|\cdot\|_{\ell^p(\ZZ^d)}$ to $\|\cdot\|_{\ell^p}$. 
\item Let $\calF$ denote the Fourier transform on $\RR^d$ defined for any function 
$f \in L^1\big(\RR^d\big)$ as
\begin{align*}
\calF f(\xi) = \int_{\RR^d} f(x) e^{2\pi i \sprod{\xi}{x}} {\: \rm d}x \quad \text{for any}\quad \xi\in\RR^d.
\end{align*}
If $f \in \ell^1\big(\ZZ^d\big)$ we define the discrete Fourier
transform by setting
\begin{align*}
\hat{f}(\xi) = \sum_{x \in \ZZ^d} f(x) e^{2\pi i \sprod{\xi}{x}} \quad \text{for any}\quad \xi\in\TT^d,
\end{align*}
where $\TT^d$ denote $d$-dimensional torus, which will be  identified with
$Q=[-1/2, 1/2]^d$.
To simplify notation we will denote by $\mathcal F^{-1}$ the inverse Fourier transform on $\RR^d$
or the inverse Fourier transform (Fourier coefficient) on the torus $\TT^d$.
It will cause no confusions and it will be always clear from the
context. 

\end{enumerate}

\section*{Acknowledgements}
The authors are grateful to the referees for careful reading of the manuscript and useful remarks
that led to the improvement of the presentation. We also thank for pointing out a simple proof of Lemma \ref{lem:6}.

\section{Estimates for the dyadic maximal function: intermediate
  scales}
\label{sec:3}
This section is intended to provide bounds independent of the
dimension for the dyadic maximal function with supremum taken over all
dyadic numbers $N$ such that $d^{1/2}\lesssim N\lesssim d$.  Since, as we discussed in the introduction 
the estimate on $\ell^2(\ZZ^d)$ for the
maximal function $\sup_{N\in \mathbb D_{C_3, \infty}}|\mathcal M_Nf|$ is covered by inequality \eqref{eq:7}, which was proved in \cite{BMSW4}.   
\begin{theorem}
	\label{thm:10}
	For every $p\in(1,\infty)$ there is a constant  $C_p>0$ independent of  $d\in\NN$ such that for every $f\in\ell^p(\ZZ^d)$ we have
        \begin{align}
          \label{eq:7}
          \big\|\sup_{N\ge cd}|\mathcal M_Nf|\big\|_{\ell^p}\le C_p\|f\|_{\ell^p},
        \end{align}
        for an absolute large constant $c>0$.
\end{theorem}
Now, in view of Theorem \ref{thm:10} our aim is to prove Theorem \ref{thm:1}.
\begin{theorem}
\label{thm:1}
Let $C_1, C_2>0$ and define $\mathbb D_{C_1, C_2}=\{N\in\mathbb D:C_1d^{1/2}\le N\le C_2d\}$. Then there exists a constant $C>0$ independent of
dimension such that for every $f\in \ell^2(\ZZ^d)$ we have
\begin{align}
  \label{eq:20}
  \big\|\sup_{N\in \mathbb D_{C_1, C_2}}\mathcal M_Nf\big\|_{\ell^2}\le C\|f\|_{\ell^2}.
\end{align}
\end{theorem}

The operator $\mathcal M_N$ is a convolution operator with the kernel
 \begin{align*}
\mathcal K_N(x)=\frac{1}{|B_N\cap \ZZ^d|}\sum_{y\in B_N\cap \ZZ^d}\delta_y(x),   
 \end{align*}
 where $\delta_y$ is the Dirac's delta at $y\in\ZZ^d$.  
In what follows for any $\xi\in\TT^d\equiv[-1/2, 1/2)^d$ we will consider the multipliers corresponding to the  operators $\mathcal M_N$, which are exponential sums given by
\begin{align}
\label{eq:116}
\mathfrak m_N(\xi)
=\hat{\mathcal K}_N(\xi)
=\frac{1}{|B_N\cap\ZZ^d|}
\sum_{x\in B_N\cap\ZZ^d}e^{2\pi i \xi\cdot x}.
\end{align}
  
For $\xi\in\TT^d$ we will write $\|\xi\|^2=\|\xi_1\|^2+\ldots+\|\xi_d\|^2$, where
$\|\xi_j\|=\dist(\xi_j, \ZZ)$ for any $j\in\NN_d$. Since we identify $\TT^d$ with $[-1/2, 1/2)^d$
hence the norm $\|\cdot\|$ coincides with the Euclidean norm $|\cdot|$ restricted to $[-1/2, 1/2)^d$.
Moreover, for every $\eta\in\TT$ we know that
$\|\eta\|\simeq|\sin(\pi\eta)|$, since $|\sin(\pi\eta)|=\sin(\pi\|\eta\|)$ and for $0\le|\eta|\le 1/2$ we have
\begin{align}
\label{eq:103}
2|\eta|\le|\sin(\pi\eta)|\le \pi|\eta|.
\end{align}

The proof of Theorem \ref{thm:1} will be based on Proposition \ref{prop:0}, which provides estimates of the multiplier $\mathfrak m_N(\xi)$ at the origin, and on Proposition \ref{prop:2}, which provides estimates of the multiplier $\mathfrak m_N(\xi)$ at infinity. Both of the estimates will be described in terms of a proportionality constant
\begin{align}
\label{eq:46}
\kappa(d, N)=Nd^{-1/2}.
\end{align}
\begin{proposition}
\label{prop:0}
For every $d, N\in\NN$ and for every $\xi\in\TT^d$ we have
\begin{align}
  \label{eq:22}
  |\mathfrak m_N(\xi)-1|\le 2\pi^2\kappa(d, N)^2\|\xi\|^2.
\end{align}
\end{proposition}
\begin{proof}
Exploiting the symmetries of $B_N\cap\ZZ^d$  we have
\begin{align}
\label{eq:43}
\begin{split}
\mathfrak m_N(\xi)&=\frac{1}{|B_N\cap\ZZ^d|}\sum_{x\in
B_N\cap\ZZ^d}\prod_{j=1}^d e^{2\pi i x_j \xi_j}\\
&=\frac{1}{|B_N\cap\ZZ^d|}\sum_{x\in B_N\cap\ZZ^d}\prod_{j=1}^d \cos(2\pi x_j \xi_j).
\end{split}
\end{align}
Recall that for any sequence $(a_j: j\in\NN_d)\subseteq\CC$ and
$(b_j: j\in\NN_d)\subseteq\CC$, if $\sup_{j\in\NN_d}|a_j|\le1$ and
$\sup_{j\in\NN_d}|b_j|\le1$ then we have
\begin{align}
\label{eq:104}
\Big|\prod_{j=1}^da_j-\prod_{j=1}^db_j\Big|\le\sum_{j=1}^d|a_j-b_j|.
\end{align}
Therefore, using \eqref{eq:104} and the formula  $\cos(2x)=\cos^2x-\sin^2x=1-2\sin^2x$, we obtain
\begin{align*}
|\mathfrak m_N(\xi)-1|&\le \frac{1}{|B_N\cap\ZZ^d|}\sum_{x\in
B_N\cap\ZZ^d}\Big|\prod_{j=1}^d \cos(2\pi x_j \xi_j)-1\Big|\\
&\le\frac{1}{|B_N\cap\ZZ^d|}\sum_{x\in
B_N\cap\ZZ^d}\sum_{j=1}^{d} |\cos(2\pi x_j
\xi_j)-1|\\
&\le\frac{2}{|B_N\cap\ZZ^d|}\sum_{x\in
B_N\cap\ZZ^d}\sum_{j=1}^{d}\sin^2(\pi x_j \xi_j).
\end{align*}
Observe  that $|\sin(\pi x y)|\le |x||\sin(\pi y)|$ for every $x\in\ZZ$ and $y\in\RR$, and observe also that
for every $i\not=j$ one has
\[
\sum_{x\in B_N\cap\ZZ^d}x_i^2=\sum_{x\in B_N\cap\ZZ^d}x_j^2=\frac{1}{d}\sum_{x\in B_N\cap\ZZ^d}|x|^2.
\]
Thus, taking  into account these observations and changing the order of summations we obtain
\begin{equation}
\label{eq:mN-1}
\begin{split}
|\mathfrak m_N(\xi)-1|&\le \frac{2}{|B_N\cap\ZZ^d|}
\sum_{x\in B_N\cap\ZZ^d}\sum_{j=1}^{d}\sin^2(\pi x_j\xi_j)\\
&\le\frac{2}{|B_N\cap\ZZ^d|}\sum_{j=1}^{d}\sin^2(\pi \xi_j)
\sum_{x\in B_N\cap\ZZ^d}x_j^2\\
&\le\frac{2}{|B_N\cap\ZZ^d|}\sum_{j=1}^{d}\sin^2(\pi \xi_j)
\frac{1}{d}\sum_{x\in B_N\cap\ZZ^d}|x|^2\\
&\le2\pi^2\kappa(d, N)^2\|\xi\|^2
  \end{split}
\end{equation}
and \eqref{eq:22} is justified. 
\end{proof}
The rest of this section is devoted to prove of Proposition
\ref{prop:2}.
\begin{proposition}
\label{prop:2}
There is a constant $C>0$ such that for any $d, N\in\NN$ if $10\le \kappa(d, N)\le 50d^{1/2}$ then   for all $\xi\in\TT^d$ we have
\begin{align}
  \label{eq:23}
  |\mathfrak m_N(\xi)|\le C\big((\kappa(d, N)\|\xi\|)^{-1}+\kappa(d, N)^{-\frac{1}{7}}\big).
\end{align}
\end{proposition}
Proposition \ref{prop:2} is essential in the proof of Theorem
\ref{thm:1}. Assume momentarily that   Proposition \ref{prop:2} has
been proven. We show how the inequalities \eqref{eq:22} and
\eqref{eq:23} can be used to deduce \eqref{eq:20}. 
\begin{proof}[Proof of Theorem \ref{thm:1}]
Since $\mathbb D_{C_1, C_2}$ is a subset of the dyadic set $\mathbb D$
we can assume, without loss of generality, that $C_1=C_2=1$ and
\eqref{eq:23} is valid when $N\in\mathbb D_{1,1}$. To complete the
proof we shall compare the averages $\mathcal M_N$ with a symmetric
diffusion semigroup on $\ZZ^d$. Namely, for every $t>0$ let $P_t$ be
the semigroup with the multiplier
\begin{align}
\label{eq:119}
\mathfrak p_t(\xi)=e^{-t\sum_{i=1}^d\sin^2(\pi\xi_i)} \quad\text{for}\quad \xi\in\TT^d.
\end{align}
It follows from a general theory for symmetric diffusion semigroups \cite{Ste1}, (see also \cite{BMSW3} for more details) that for every $p\in(1, \infty)$ there is $C_p>0$ independent of $d\in\NN$ such that for every $f\in\ell^p(\ZZ^d)$ we have
\begin{align}
\label{eq:47}
\big\|\sup_{t>0}|P_tf|\big\|_{\ell^p}\le C_p \|f\|_{\ell^p}.
\end{align}
Hence \eqref{eq:47} reduces the proof of  \eqref{eq:20} to the dimension-free estimate on  $\ell^2(\ZZ^d)$ for the following square function
\[
Sf(x)=\Big(\sum_{N\in \mathbb D_{C_1, C_2}}|\mathcal M_Nf(x)-P_{N^2/d}f(x)|^2\Big)^{1/2} \quad\text{for} \quad x\in\ZZ^d.
\]
By Plancherel's formula,
\eqref{eq:22} and \eqref{eq:23}, we have
\begin{align*}
\|S(f)\|_{\ell^2}^2&\leq \int_{\TT^d}\,
\bigg(\sum_{\substack{m\in\ZZ:\\d^{1/2}\le 2^m\le d}}\min\big\{2^{2m}\|\xi\|^2/d,(2^{2m}\|\xi\|^2/d)^{-1}\big\}+d^{1/7}\sum_{\substack{m\in\ZZ:\\d^{1/2}\le 2^m\le d}}
2^{-2m/7}\bigg) |\hat{f}(\xi)|^2{\rm d}\xi\\
&\leq C\,\|f\|_{\ell^2}^2.
\end{align*}
This completes the proof of Theorem \ref{thm:1}.
\end{proof}

\subsection{Some preparatory estimates}

The proof of Proposition \ref{prop:2} will require some  bunch of lemmas, which will be based on the following  precise version of Stirling's formula \cite{Rob}. For every $m\in\NN$ one has
\begin{align}
\label{eq:105}
\sqrt{2\pi}\:m^{m+1/2}e^{-m}e^{\frac{1}{12m+1}}\le m!\le \sqrt{2\pi}\:m^{m+1/2}e^{-m}e^{\frac{1}{12m}}.
\end{align}
We shall need the following crude size estimates for the number of lattice points in $B_N$.
\begin{lemma}
\label{lem:4}
	For all $d, N\in \NN$  we have
	\[
        (2\lfloor \kappa(d, N)\rfloor+1)^d\leq |B_N\cap\ZZ^d|\leq
       (2\pi e)^{d/2}(\kappa(d, N)^2+1/4)^{d/2}.
        \]
\end{lemma}
\begin{proof}
The lower bound follows from the inclusion
$[-\kappa(d, N), \kappa(d, N)]^d\cap\ZZ^d\subseteq B_N\cap\ZZ^d$. To
prove the upper bound we use \cite[Lemma 5.1]{BMSW4} to obtain that
\begin{align*}
|B_N\cap\ZZ^d|\leq  2|B_{(N^2+d/4)^{1/2}}|=\frac{2\pi^{d/2}}{\Gamma(d/2+1)}(N^2+d/4)^{d/2}
= \frac{2\pi^{d/2} d^{d/2}}{\Gamma(d/2+1)} (\kappa(d, N)^2+1/4)^{d/2}.
\end{align*}

Assume that $d=2m$ is even and note that by \eqref{eq:105} we have
\begin{align*}
  \frac{\pi^{d/2}}{\Gamma(d/2+1)}=\frac{\pi^m}{m!}\le
  \frac{\pi^me^m}{(2\pi)^{1/2}m^{m+1/2}}\le\frac{(2\pi e)^{d/2}}{2d^{d/2}}.
\end{align*}
If $d=1$ then $\frac{\pi^{1/2}}{\Gamma(3/2)}=2\le (2\pi e)^{1/2}/2$. 
Assume now that $d=2m+1\ge3$ is odd and note 
\begin{align*}
  \frac{\pi^{d/2}}{\Gamma(d/2+1)}=\frac{2^{2m+1}\pi^mm!}{(2m+1)!}
\le 
\frac{2^{d}\pi^{(d-1)/2}m^{m+1/2}e^{d-m}e^{{1}/{12}}}{d^{d+1/2}}\le \bigg(\frac{e^{1/6}}{3\pi}\bigg)^{1/2}\frac{(2\pi e)^{d/2}}{d^{d/2}}\le\frac{(2\pi e)^{d/2}}{2d^{d/2}}.
\end{align*}
The proof of the lemma is completed.
\end{proof}
We shall also need  balls in lower dimensions. For every
$r\in\NN_d$ let $B_R^{(r)}$ denote the Euclidean ball in $\RR^r$
centered at the origin with radius $R>0$. We now use Lemma \ref{lem:4} to control size of certain error subsets of $B_N\cap\ZZ^d$, which tells us, to some extent, that a significant amount of mass of $B_N\cap\ZZ^d$ is concentrated near its boundary, like in the continuous setup.  
\begin{lemma}
\label{lem:5}
Given $\varepsilon_1, \varepsilon_2\in(0, 1]$ we define for every $d, N\in\NN$ the set
\[
E=\big\{x\in B_N\cap\ZZ^d\colon |\{i\in\NN_d\colon |x_i|\ge
\varepsilon_2\kappa(d, N)\}|\le\varepsilon_1 d\big\}.
\]
If $\varepsilon_1, \varepsilon_2\in(0, 1/10]$ and $\kappa(d,N)\ge10$, then we have
\begin{align}
\label{eq:84}
|E| \le 2e^{-\frac{d}{10}}|B_N\cap\ZZ^d|.
\end{align}
\end{lemma}
\begin{proof}
Note that if $x\in E,$ then there is  $I\subseteq\NN_d$ such that
$|I|\le\varepsilon_1 d$ and $|x_i|\ge \varepsilon_2\kappa(d, N)$ precisely
when $i\in I$. Therefore, we have
\[
E\subseteq \bigcup_{I\subseteq \NN_d:|I|\le\varepsilon_1d}\{x\in
B_N\cap\ZZ^d\colon |x_i|\ge \varepsilon_2\kappa(d, N)\ \text{ precisely when
} i\in I\}.
\]
For  $I$ as above we have
$|x_i|\le\varepsilon_2\kappa(d, N)$ for every $i\in \NN_d\setminus I$ and consequently
\[
|\{x\in
B_N\cap\ZZ^d\colon |x_i|\ge \varepsilon_2\kappa(d, N)\ \text{ precisely when
} i\in I\}|\le (2\varepsilon_2\kappa(d, N)+1)^{d-|I|}\big|B_N^{(|I|)}\cap\ZZ^{|I|}\big|.
\]
Hence
\begin{align}
\label{eq:24}
  |E|\le \sum_{I\subseteq \NN_d:|I|\le\varepsilon_1d}(2\varepsilon_2\kappa(d, N)+1)^{d-|I|}\big|B_N^{(|I|)}\cap\ZZ^{|I|}\big|.  
\end{align}
	Clearly, there are $d\choose m$ subsets $I$ of $\NN_d$ of size
        $m$. We use the upper bound from Lemma \ref{lem:4} (with $d=m$) to  estimate \eqref{eq:24} and obtain
\begin{align*}
|E|\le & (2\varepsilon_2\kappa(d, N)+1)^{d}+ \sum_{1\le m\le \varepsilon_1d}{{d}\choose{m}}(2\varepsilon_2\kappa(d,N)+1)^{d-m}\big|B_N^{(m)}\cap\ZZ^{m}\big|\\
\le &(2\varepsilon_2\kappa(d, N)+1)^{d}+\sum_{1\le m\le\varepsilon_1d}\frac{d^m}{m!}
(2\varepsilon_2\kappa(d, N)+1)^{d-m}(2\pi e)^{{m}/{2}}\big(\kappa(d, N)^2d/m+1/4\big)^{{m}/{2}}\\
\le & (2\varepsilon_2\kappa(d,N)+1)^{d}
+\sum_{1\le m\le\varepsilon_1d}\bigg(\frac{d}{m}\bigg)^{{3m}/{2}}(2\varepsilon_2\kappa(d, N)+1)^{d-m}
(2\pi e^3)^{m/2}\big(\kappa(d, N)^2+m/(4d)\big)^{m/2},
\end{align*}
where in the last line we have used that $\frac{1}{m!}\le \frac{e^m}{m^m}$. Therefore, using the lower bound from  Lemma \ref{lem:4}, we obtain 
\begin{align}
\label{eq:45}
\begin{split}
|E|\le &(2\varepsilon_2\kappa(d,N)+1)^{d}
+\sum_{1\le m\le\varepsilon_1d}\bigg(\frac{2^{1/3}ed}{m}\bigg)^{{3m}/{2}}(2\varepsilon_2\kappa(d, N)+1)^{d-m}
\big(\pi^{1/2}\kappa(d, N)+1\big)^{m}\\
\le &\bigg(\bigg(\frac{2\varepsilon_2\kappa(d,N)+1}{2\lfloor\kappa(d,N)\rfloor+1}\bigg)^{d}
+\sum_{1\le m\le\varepsilon_1d}\bigg(\frac{2^{1/3}ed}{m}\bigg)^{{3m}/{2}}
\bigg(\frac{2\varepsilon_2\kappa(d,N)+1}{2\lfloor\kappa(d,N)\rfloor+1}\bigg)^{d-m}
\bigg)|B_N\cap\ZZ^{d}|\\
\le & \Big(e^{-\frac{16d}{19}}
+e^{-\frac{72d}{95}}\sum_{m=1}^{ d_0}
e^{\varphi(m)}
\Big)|B_N\cap\ZZ^{d}|,
\end{split}
\end{align}
where $d_0=\lfloor \varepsilon_1d\rfloor$ and $\varphi(x)=\frac{3x}{2}\log\big(\frac{2^{1/3}ed}{x}\big)$.
In the last inequality we  have used that $\kappa(d, N)\ge10$ and $\varepsilon_1,\varepsilon_2\le1/10$ and the following bound $\frac{2\varepsilon_2\kappa(d,N)+1}{2\lfloor\kappa(d,N)\rfloor+1}\le \frac{2\varepsilon_2\kappa(d,N)+1}{2\kappa(d,N)-1}=1-\frac{2(1-\varepsilon_2)\kappa(d,N)-2}{2\kappa(d,N)-1}\le e^{-\frac{16}{19}}$. Note that  
$\big(0, d/10\big]\ni x\mapsto \varphi(x)={\frac{3x}{2}\log\big(\frac{2^{1/3}ed}{x}\big)}$ is increasing, since
\[
\varphi'(x)=\frac{3}{2}\log\bigg(\frac{2^{1/3}ed}{x}\bigg)-\frac{3}{2}\ge\log3.
\]
Thus
\begin{align*}
\sum_{m=1}^{d_0}e^{\varphi(m)}
\le e^{\varphi(d_0)}\sum_{m=0}^{d_0}e^{-(d_0-m)\log3}
\le \frac{3}{2}e^{\varphi(d_0)}
\le \frac{3}{2}e^{\frac{3\log(2^{1/3}\cdot10\cdot e)d}{20}}
\le \frac{3}{2}e^{\frac{3d}{5}},
\end{align*}
and consequently 
\begin{align}
\label{eq:65}
e^{-\frac{16d}{19}}
+e^{-\frac{72d}{95}}\sum_{m=1}^{ d_0}
e^{\varphi(m)}
\le e^{-\frac{16d}{19}}+ \frac{3}{2}e^{-\frac{3d}{19}}
\le e^{-\frac{d}{10}}\Big(e^{-\frac{141}{190}}+\frac{3}{2}\Big)
\le2e^{-\frac{d}{10}}.
\end{align}
Combining \eqref{eq:45} with \eqref{eq:65} we obtain \eqref{eq:84}.
\end{proof}

\subsection{Analysis on  permutation groups}

We have to fix more notation and terminology. Let ${\rm Sym}(d)$ be the permutation group on $\NN_d$. We will write $\sigma\cdot x=(x_{\sigma(1)}, \ldots, x_{\sigma(d)})$ for every $x\in\RR^d$ and $\sigma\in{\rm Sym}(d)$.  Later on  $\mathbb P$ will denote the uniform distribution on the symmetry group ${\rm Sym}(d)$, i.e. $\mathbb P(A)={|A|}/{d!}$ for any $A\subseteq{\rm Sym}(d)$, since we know that $|{\rm Sym}(d)|=d!$. The expectation $\mathbb E$ will be always taken with respect to the uniform distribution $\mathbb P$ on the symmetry group
${\rm Sym}(d)$. The next few lemmas will rely on the properties of the permutation group.

\begin{lemma}
\label{lem:6}
Assume that  $I, J\subseteq\NN_d$ and 
$|J|=r$ for some $0\le r\le d$. Then
\begin{align}
\label{eq:102}
\mathbb P[\{\sigma\in{\rm Sym}(d)\colon |\sigma(I)\cap J|\le {r|I|}/{(5d)}\}]\le  e^{-\frac{r|I|}{10d}}.
\end{align}
 In particular, if $\delta_1, \delta_2\in(0, 1]$ satisfy
$5\delta_2\le\delta_1$ and $\delta_1d\le |I|\le d$, then we have
\begin{align}
  \label{eq:27}
  \mathbb P[\{\sigma\in{\rm Sym}(d)\colon |\sigma(I)\cap
  J|\le \delta_2 r\}]\le e^{-\frac{\delta_1r}{10}}.
\end{align}  
\end{lemma}

\begin{proof}
Inequality \eqref{eq:27} is a consequence of \eqref{eq:102}. To prove
\eqref{eq:102} we fix $0\le r\le d$ and $I, J\subseteq\NN_d$ such that
$|J|=r$. It is not difficult to see that
\[
\mathbb P[\{\sigma\in{\rm Sym}(d)\colon |\sigma(I)\cap J|=k\}]=\binom{r}{k}\binom{d-r}{|I|-k}\binom{d}{|I|}^{-1},
\]
which means that the random variable $X=|\sigma(I)\cap J|$ has the hypergeometric distribution. Appealing to the Hoeffding type inequality \cite[Theorem 2.10 and inequality (2.6)]{JLR}, we obtain for every $\tau\in(0, 1)$ that
\begin{align*}
\mathbb P[\{X\le (1-\tau)\mathbb E[X]\}]\le e^{-\tau^2\mathbb E[X]/2}.
\end{align*}
Taking $\tau=4/5$ in this inequality and noting that $\mathbb E[X]=r|I|/d$ we conclude that
\begin{align*}
\mathbb P[\{\sigma\in{\rm Sym}(d)\colon |\sigma(I)\cap J|\le {r|I|}/{(5d)}\}]\le  e^{-\frac{8r|I|}{25d}}.
\end{align*}
This completes the proof of the lemma. 
\end{proof}

\begin{lemma}
\label{lem:7}
Assume that we have a finite decreasing sequence 
$0\le u_d\le\ldots\le u_2\le u_1\le(1-\delta_0)/2$ for some
$\delta_0\in(0, 1)$. Suppose that $I\subseteq\NN_d$ satisfies
$\delta_1d\le |I|\le d$ 
for some $\delta_1\in(0, 1]$. Then for
every $J=(d_0, d]\cap\ZZ$ with $0\le d_0\le d$ we have  
\begin{align}
  \label{eq:30}
  \mathbb E\Big[\exp\Big({-\sum_{j\in\sigma(I)\cap J}u_j}\Big)\Big]
  \le 3\exp\Big({-\frac{\delta_0\delta_1}{20}\sum_{j\in J}u_j}\Big).
\end{align}
\end{lemma}
\begin{proof}
Let us define $\bar{m}=\sum_{j\in J}u_j$, and without loss of
generality, we may assume that $\bar{m}\ge1$, since otherwise
\eqref{eq:30} is obvious. We now take  $m\in\NN$ such that $m\le
\bar{m}<m+1$, with this choice of $m$, we define 
\[
U_k=(d-
(k+1)m, d-km]
\]
for every $0\le k\le k_0$, where $k_0$ is the maximal integer such that $U_{k_0}\cap J\not=\emptyset$.  Observe that  
\begin{align}
  \label{eq:31}
  \sum_{j\in\sigma(I)\cap J}u_j\ge\sum_{k=0}^{k_0-1}\sum_{j\in \sigma(I)\cap J\cap U_k}u_j\ge\sum_{k=0}^{k_0-1}u_{d-km}|\sigma(I)\cap J\cap U_k|.
\end{align}
We shall prove that
\begin{align}
  \label{eq:32}
  \delta_0\le \sum_{k=0}^{k_0-1}u_{d-km}\le 2.
\end{align}
Indeed, on the one hand, we have 
\begin{align*}
  \sum_{k=0}^{k_0-1}u_{d-km}\ge\frac{1}{m}\sum_{k=1}^{k_0-1}\sum_{j\in U_{k-1}}u_j
  \ge\frac{1}{m}\Big(\bar{m}-\sum_{j\in U_{k_0}}u_j-\sum_{j\in U_{k_0-1}}u_j\Big)\ge\frac{1}{m}\big(m-(1-\delta_0)m\big)\ge\delta_0.
\end{align*}
On the other hand, we have
\begin{align*}
  \sum_{k=0}^{k_0-1}u_{d-km}\le\frac{1}{m}\sum_{k=0}^{k_0-1}\sum_{j\in U_{k}}u_j\le\frac{\bar{m}}{m}\le2.
  \end{align*}
  which proves \eqref{eq:32}.
Let $s=\sum_{k=0}^{k_0-1}u_{d-km}$ and we note that
\eqref{eq:31}  yields
\begin{align}
  \label{eq:33}
  \begin{split}
  \mathbb E\Big[\exp\Big({-\sum_{j\in\sigma(I)\cap J}u_j}\Big)\Big]
  &\le\mathbb E\Big[\exp\Big({-s\sum_{k=0}^{k_0-1}\frac{u_{d-km}}{s}|\sigma(I)\cap J\cap U_k|}\Big)\Big]\\
  &\le \sum_{k=0}^{k_0-1}\frac{u_{d-km}}{s}\mathbb E\big[e^{-s|\sigma(I)\cap J\cap U_k|}\big]\\
  &\le\sup_{0\le k\le k_0-1}\mathbb E\big[e^{-s|\sigma(I)\cap J\cap U_k|}\big],
  \end{split}
\end{align}
where in the second inequality of \eqref{eq:33} we have used
convexity. Take $\delta_2\in(0, 1)$ such that $\delta_2=\delta_1/5$ and define $A_m=\{\sigma\in{\rm Sym}(d)\colon |\{\sigma(I)\cap J\cap U_k\}|\le\delta_2m\}$. Invoking \eqref{eq:27}, with $U_k\cap J$ in place of $J$ and $r=m$, since $U_k\cap J=U_k$ for any $0\le k\le k_0-1$,  we see that
\begin{align*}
  \mathbb E\big[e^{-s|\sigma(I)\cap J\cap U_k|}\big]\le
  &\mathbb E\big[e^{-s|\sigma(I)\cap J\cap U_k|}\ind{A_m^{\bf c}}\big]+\mathbb E\big[e^{-s|\sigma(I)\cap J\cap U_k|}\ind{A_m}\big]\\
 \le &e^{-s\delta_2m}
 +\mathbb P[\{\sigma\in{\rm Sym}(d)\colon |\{\sigma(I)\cap J\cap U_k\}|\le\delta_2m\}]\\
\le &3\exp\Big({-\frac{\delta_0\delta_1}{20}\sum_{j\in J}u_j}\Big)
\end{align*}
due to \eqref{eq:32} and $m\ge\frac{1}{2}\sum_{j\in J}u_j$. This completes the proof of Lemma \ref{lem:7}.
\end{proof}

\subsection{A decrease dimension trick}

Choosing $r$ large, but sufficiently small compared with $\kappa(d, N)$, one may then perform finer estimates by exploiting the distribution of lattice points in the ball $B_{\sqrt{l}}^{(r)}$ with $l\gtrsim\kappa(d, N)^2r$.
The next two lemmas will allow us to reduce the matters to the lower dimensional case, where things are simpler. For $r\in\NN$ let $S_{\sqrt{R}}^r=\{x\in\RR^r\colon |x|^2=R\}$ be the sphere in $\RR^r$ centered at the origin and radius $\sqrt{R}>0$.

\begin{lemma}
\label{lem:8}
For $d, N\in\NN$, $\varepsilon\in(0, 1/50]$ and an integer $1\le r\le d$  we define
\[
E=\{x\in B_N\cap\ZZ^d\colon \sum_{i=1}^rx_i^2<\varepsilon^3\kappa(d, N)^2r\}.
\]
If $\kappa(d, N)\ge10$, then we have
\begin{align}
  \label{eq:37}
  |E|\le 4e^{-\frac{\varepsilon r}{10}}|B_N\cap\ZZ^d|.
\end{align}
As a consequence, there exists $E'\subseteq E$, such that
\begin{align}
  \label{eq:38}
  \begin{split}
  B_N\cap\ZZ^d&=\bigcup_{l\ge\varepsilon^3\kappa(d, N)^2r}\big(B_{\sqrt{l}}^{(r)}\cap\ZZ^r\big)
                \times\big(S^{d-r}_{\sqrt{n-l}}\cap\ZZ^{d-r}\big)\cup E'\qquad \text{with}\qquad n=N^2.
  \end{split}
  \end{align}
 \end{lemma}
\begin{proof}
Let $\delta_1\in(0, 1/10]$ be  such that  $\delta_1\ge5\varepsilon$, and define $I_x=\{i\in\NN_d: |x_i|\ge\varepsilon\kappa(d, N)\}$. We have
$E\subseteq E_1\cup E_2$, where
\begin{align*}
E_1&=\{x\in B_N\cap\ZZ^d\colon \sum_{i\in I_x\cap\NN_r}x_i^2<\varepsilon^3\kappa(d, N)^2r\ \text{ and }\ |I_x|\ge\delta_1d\},\\
E_2&=\{x\in B_N\cap\ZZ^d\colon |I_x|<\delta_1d\}.
\end{align*}
By Lemma \ref{lem:5} (with $\varepsilon_1=\delta_1$ and $\varepsilon_2=\varepsilon$) we have $|E_2|\le2e^{-\frac{d}{10}}|B_N\cap\ZZ^d|$, provided that $\kappa(d, N)\ge10$. Observe that
\begin{align*}
|E_1|&=\sum_{x\in B_N\cap\ZZ^d}\frac{1}{d!}\sum_{\sigma\in{\rm Sym}(d)}\ind{E_1}(\sigma^{-1}\cdot x)\\
&=\sum_{x\in B_N\cap\ZZ^d}\mathbb P[\{\sigma\in{\rm Sym}(d)\colon
\sum_{i\in \sigma(I_x)\cap\NN_r}x_{\sigma^{-1}(i)}^2<\varepsilon^3\kappa(d, N)^2r\ \text{ and }\ |\sigma(I_x)|\ge\delta_1d\}],
\end{align*}
since $I_{\sigma^{-1}\cdot x}=\sigma(I_x)$. 
Now by Lemma \ref{lem:6} (with $J=\NN_r$, $\delta_2=\frac{\delta_1}{5}$ and $\delta_1$ as above) we obtain, for every $x\in B_N\cap\ZZ^d$, that
\begin{align*}
  \mathbb P[\{\sigma\in&{\rm Sym}(d)\colon
\sum_{i\in \sigma(I_x)\cap\NN_r}x_{\sigma^{-1}(i)}^2<\varepsilon^3\kappa(d, N)^2r\ \text{ and }\ |\sigma(I_x)|\ge\delta_1d\}]\\
  &\le
\mathbb P[\{\sigma\in{\rm Sym}(d)
  \colon |\sigma(I_x)\cap\NN_r|\le\delta_2r\}]\le 2e^{-\frac{\delta_1 r}{10}},
\end{align*}
since
\[
\{\sigma\in{\rm Sym}(d)
\colon \sum_{i\in \sigma(I_x)\cap\NN_r}x_{\sigma^{-1}(i)}^2<\varepsilon^3\kappa(d, N)^2r\ \text{ and }\
|\sigma(I_x)|\ge\delta_1d \ \text{ and }\ 
|\sigma(I_x)\cap\NN_r|>\delta_2r\}=\emptyset.
  \]
  Thus $|E_1|\le 2e^{-\frac{\varepsilon r}{2}}$, which proves \eqref{eq:37}. To prove \eqref{eq:38} we write
  \[
B_N\cap\ZZ^d=\bigcup_{l=0}^n\big(B_{\sqrt{l}}^{(r)}\cap\ZZ^r\big)
\times\big(S^{d-r}_{\sqrt{n-l}}\cap\ZZ^{d-r}\big).
  \]
  Then we see that
  \[
\Big(\bigcup_{l=0}^n\big(B_{\sqrt{l}}^{(r)}\cap\ZZ^r\big)
\times\big(S^{d-r}_{\sqrt{n-l}}\cap\ZZ^{d-r}\big)\Big)\cap E^{\bf c}
=\Big(\bigcup_{l\ge\varepsilon^3\kappa(d, N)^2r}\big(B_{\sqrt{l}}^{(r)}\cap\ZZ^r\big)
\times\big(S^{d-r}_{\sqrt{n-l}}\cap\ZZ^{d-r}\big)\Big)\cap E^{\bf c},
\]
where $n=N^2$, and consequently we obtain \eqref{eq:38} with some $E'\subseteq E$.
The proof is completed.                
\end{proof}

We shall need the lower dimensional multipliers 
\begin{align}
\label{eq:21}
\mathfrak m_{R}^{(r)}(\eta)=\frac{1}{|B_{R}^{(r)}\cap\ZZ^d|}\sum_{x\in
	B_{R}^{(r)}\cap\ZZ^d}e^{2\pi i \eta\cdot x},\qquad \eta\in \TT^r,
\end{align}
where $r\in \NN$ and $R>0.$
\begin{lemma}
\label{lem:9}
For $d, N\in\NN$ and $\varepsilon\in(0, 1/50]$ if $\kappa(d, N)\ge10$, then for every $1\le r\le d$ and $\xi\in\TT^d$ we have
\begin{align}
  \label{eq:39}
  |\mathfrak m_N(\xi)|\le\sup_{l\ge\varepsilon^3\kappa(d, N)^2r}|\mathfrak
  m_{\sqrt{l}}^{(r)}(\xi_1,\ldots, \xi_r)|+4e^{-\frac{\varepsilon r}{10}}.
\end{align}

\end{lemma}
\begin{proof}
We identify $\RR^d\equiv\RR^r\times\RR^{d-r}$ and $\TT^d\equiv\TT^r\times\TT^{d-r}$ and we will write $\RR^d\ni x=(x^1, x^2)\in \RR^r\times\RR^{d-r}$ and $\TT^d\ni \xi=(\xi^1, \xi^2)\in \TT^r\times\TT^{d-r}$ respectively.
Invoking \eqref{eq:38} one obtains
\begin{align*}
  |\mathfrak m_N(\xi)|&\le\frac{1}{|B_N\cap\ZZ^d|}\sum_{l\ge\varepsilon^3\kappa(d, N)^2r}\;\sum_{x^2\in
  S_{\sqrt{n-l}}^{d-r}\cap\ZZ^{d-r}}|B_{\sqrt{l}}^{(r)}\cap\ZZ^r|\frac{1}{|B_{\sqrt{l}}^{(r)}\cap\ZZ^r|}\Big|\sum_{x^1\in
  B_{\sqrt{l}}^{(r)}\cap\ZZ^r}e^{2\pi i \xi^1\cdot
                        x^1}\Big|+4e^{-\frac{\varepsilon r}{10}}\\
&  \le\sup_{l\ge\varepsilon^3\kappa(d, N)^2r}|\mathfrak
  m_{\sqrt{l}}^{(r)}(\xi_1,\ldots, \xi_r)|+4e^{-\frac{\varepsilon r}{10}}.
\end{align*}
In the last inequality the disjointness in the decomposition from \eqref{eq:38} has been used.
\end{proof}

Lemma \ref{lem:9} will play an essential role in the proof of Proposition \ref{prop:2}. The decrease of the dimension will allow us to approximate the resulting multiplier \eqref{eq:39} by its continuous counterpart with a dimension-free error term.
In order to control the error term efficiently we will need the following two simple lemmas. 

\begin{lemma}
\label{lem:10}
Let $R\ge1$ and let $r\in\NN$ be  such that   $r\le R^{\delta}$ 
 for some $\delta\in(0, 2/3)$. Then for every $z\in\RR^r$ we have
\begin{align}
  \label{eq:40}
  \big||(z+B_R^{(r)})\cap\ZZ^r|-|B_R^{(r)}|\big|\le |B_R^{(r)}|r^{3/2}R^{-1}e^{r^{3/2}/R}\le e|B_R^{(r)}|R^{-1+3\delta/2}.
\end{align}

\end{lemma}
\begin{proof}
	Throughout the proof we abbreviate $B_R=B_R^{(r)}.$ 
Observe that
\begin{align*}
  |(z+B_R)\cap\ZZ^r|=\sum_{x\in
                  \ZZ^r} \int_{Q}\ind{z+B_R}(x){\rm d}y
                \le\sum_{x\in
                  \ZZ^r} \int_{Q}\ind{z+B_{R+r^{1/2}}}(x+y){\rm d}y=|z+B_{R+r^{1/2}}|,
\end{align*}
and
\[
|B_{R+r^{1/2}}|=|B_R|\bigg(1+\frac{r^{1/2}}{R}\bigg)^r\le
e^{r^{3/2}/R}|B_R|\le |B_R|\big(1+r^{3/2}R^{-1}e^{r^{3/2}/R}\big),
\]
since $e^x\le(1+xe^x)$. Arguing in a similar way we obtain
\begin{align*}
  |(z+B_R)\cap\ZZ^r|=\sum_{x\in
                  \ZZ^r} \int_{Q}\ind{z+B_R}(x){\rm d}y
                \ge\sum_{x\in
                  \ZZ^r} \int_{Q}\ind{z+B_{R-r^{1/2}}}(x+y){\rm d}y=|z+B_{R-r^{1/2}}|,
\end{align*}
and
\[
|B_{R-r^{1/2}}|=|B_R|\bigg(1-\frac{r^{1/2}}{R}\bigg)^r\ge
|B_R|\big(1-r^{3/2}R^{-1}\big).
\]
These inequalities imply \eqref{eq:40}, since $r\le R^{\delta}$.
\end{proof}

\begin{lemma}
\label{lem:11}
Let $R\ge1$ and let $r\in\NN$ be such that $r\le R^{\delta}$ for some $\delta\in(0,
2/3)$. Then for  
every $z\in\RR^r$ we have
\begin{align}
  \label{eq:41}
  \begin{split}
  \big|\big(B_{R}^{(r)}\cap\ZZ^r\big)\triangle\big((z+B_{R}^{(r)})\cap\ZZ^r\big)\big|&\le
4e\big(r|z|R^{-1}e^{r|z|R^{-1}}+e^{r|z|R^{-1}}R^{-1+3\delta/2}\big)|B_R^{(r)}|\\
  &\le
  4e\big(|z|R^{-1+\delta}e^{|z|R^{-1+\delta}}+
  e^{|z|R^{-1+\delta}}R^{-1+3\delta/2}\big)|B_R^{(r)}|.
  \end{split}
\end{align}
\end{lemma}

\begin{proof}
	We again abbreviate $B_R=B_R^{(r)}.$
Observe that
\[
\big(B_R\cap\ZZ^r\big)\setminus\big((z+B_R)\cap\ZZ^r\big)\subseteq
\big((z+B_{R+|z|})\cap\ZZ^r\big)\setminus\big((z+B_R)\cap\ZZ^r\big). 
\]
Thus by Lemma \ref{lem:10} one has
\begin{align*}
  \big|\big((z+B_{R+|z|})\cap\ZZ^r\big)\setminus\big((z&+B_R)\cap\ZZ^r\big)\big|=
                                                                             \big|\big((z+B_{R+|z|})\cap\ZZ^r\big)\big|-\big|\big((z+B_R)\cap\ZZ^r\big)\big|\\
                                                                           \le|B_{R+|z|}|&-|B_{R}|+e|B_{R+|z|}|(R+|z|)^{-1+3\delta/2}+e|B_{R}|R^{-1+3\delta/2}\\
\le|B_{R+|z|}|&-|B_{R}|+2e|B_{R+|z|}|R^{-1+3\delta/2}\\
\le|B_{R+|z|}|&-|B_{R}|+2e|B_R| e^{r|z|R^{-1}}R^{-1+3\delta/2}.
\end{align*}
We also see
\begin{align*}
  |B_{R+|z|}|-|B_{R}|&=\big((1+|z|R^{-1})^r-1\big)|B_R|\\
                     &\le r|z|R^{-1}(1+|z|R^{-1})^r|B_R|  \\
  &\le r|z|R^{-1}e^{r|z|R^{-1}}|B_R|.  
\end{align*}
Hence
\[
\big|\big(B_R\cap\ZZ^r\big)\setminus\big((z+B_R)\cap\ZZ^r\big)\big|\le
\big(r|z|R^{-1}e^{r|z|R^{-1}}+2e\cdot e^{r|z|R^{-1}}R^{-1+3\delta/2}\big)|B_R|.
\]
We obtain the same bound for
$\big|\big((z+B_R)\cap\ZZ^r\big)\setminus\big(B_R\cap\ZZ^r\big)\big|$
and this gives \eqref{eq:41}.
\end{proof}

We now recall the dimension-free estimates of the Fourier transform for the multiplies associated with averaging operators \eqref{eq:93} in $\RR^d$. For a symmetric convex body $G\subset \RR^d$ we let the multipliers 
\[
m^G(\xi)=\frac{1}{|G|}\calF(\ind{G})(\xi)\quad\text{for}\quad \xi\in \RR^d.
\] 
It is easy to see that $m^G(t\xi)$ is the multiplier corresponding to the operator $M_t^G$ from \eqref{eq:93}.

Assume that $|G|=1$ and that $G$ is
in the isotropic position, which means that there is an isotropic
constant $L=L(G)>0$ such that for every unit vector
$\xi\in\RR^d$ we have
\begin{align}
\label{eq:25}
\int_G (\xi \cdot x)^2{\rm d}x= L(G)^2.
\end{align}
Then the kernel of the averaging operator \eqref{eq:93} satisfies
\[
|G_R|^{-1}\ind{G_R}(x)=R^{-d}\ind{G}(R^{-1} x)
\]
for all $R>0$, and the multiplier satisfies 
\[
\calF (|G_R|^{-1}\ind{G_R})(\xi)=m^G(R\xi)=\calF(\ind{G})(R\xi).
\]
The isotropic position of $G$ allows us to provide dimension-free
estimates for the multiplier $m^G$.

\begin{theorem}[{\cite[eq. (10),(11),(12)]{B1}}]
\label{thm:3}
	Given  a symmetric convex body  $G\subset \RR^d$ with volume one,
	which is in the isotropic
	position, there exists a
	constant $C>0$ such that for every $\xi\in\RR^d$ we
	have
	\begin{align}
\label{eq:28}
	|m^G(\xi)|\leq C(L |\xi|)^{-1},\qquad  |m^G(\xi)-1|\leq CL
	|\xi|,\qquad   |\langle\xi,\nabla m^G(\xi)\rangle|\le C.
	\end{align}
	The constant $L=L(G)$ is defined  in \eqref{eq:25}, while $C$ is a
	universal constant which does not depend on $d$.

\end{theorem}
The following lemma is a simple consequence of  Theorem \ref{thm:3}.
\begin{lemma}
	\label{lem:19}
	For every  $q\in [1, \infty]$ there is a constant $c_q>0$  independent of $d$ and such that for every $R>0$ and $\xi\in\RR^d$ we have  
	\begin{equation}
\label{eq:29}
	|m^{B_R^q}(\xi)|\leq C (c_q R d^{-1/q} |\xi|)^{-1},
        \quad \text{ and }\quad
        |m^{B_R^q}(\xi)-1|\leq C (c_q Rd^{-1/q} |\xi|),
	\end{equation}
	with the implied constant $C>0$ independent of $d$.
\end{lemma}

\begin{proof}
	Let $s=s(d, q)$ be such that $|B_s^q|=s^d|B^q|=1.$ It is justified in \cite[Section 3]{Mul1} that $s\simeq a_q d^{1/q}$ and that  $B_s^q$ is in isotropic position with $L(B_s^q)\simeq A_q$. Here the constants $a_q$ and $A_q$ depend only on the parameter $q$. Therefore, from  \eqref{eq:28} it follows that
	 \begin{equation}	
\label{eq:92}
|m^{B_s^q}(\xi)|\leq C (A_q |\xi|)^{-1},
\quad \text{ and }\quad
|m^{B_s^q}(\xi)-1|\leq C (A_q  |\xi|).
\end{equation}
	 Changing variables  we see that
	 \begin{align*}
         m^{B_R^q}(\xi)=\frac{1}{|B_R^q|}\int_{B_R^q} e^{2\pi ix\cdot \xi}\dif x
         =\frac{1}{|B_s^q|}\int_{B_s^q} e^{2\pi i y\cdot ({R}{s^{-1}}\xi)}\dif y
         =m^{B_s^q}(Rs^{-1}\xi).
         \end{align*}
	In view of the above equality, recalling that $s^{-1}\simeq a_q^{-1} d^{-1/q}$ and using \eqref{eq:92} we are lead to \eqref{eq:29} with $c_q=A_qa_q^{-1}$ and the proof is completed. 
\end{proof}

\begin{lemma}
\label{lem:20}
There exists a constant $C>0$ such that for every  $\delta\in(0, 1/2)$ and for all
$r\in\NN$ and $R>0$ satisfying $1\le r\le R^{\delta}$   we have  
\begin{align*}
|\mathfrak m_R^{(r)}(\eta)|\le C\big(\kappa(r, R)^{-\frac{1}{3}+\frac{2\delta}{3}}
+r\kappa(r, R)^{-\frac{1+\delta}{3}}+\big(\kappa(r, R)\|\eta\|\big)^{-1}\big)
\end{align*}
for every $\eta\in\TT^r$, where $\mathfrak m_R^{(r)}(\eta)$ is the multiplier from \eqref{eq:21}.
\end{lemma}
\begin{proof}
The inequality is obvious when $R\le 16,$ hence it suffices to consider $R> 16.$  	
	
Firstly, we assume that
$\max\{\|\eta_1\|,\ldots,\|\eta_r\|\}>\kappa(r, R)^{-\frac{1+\delta}{3}}$.  Let $M=\big\lfloor
\kappa(r, R)^{\frac{2-\delta}{3}}\big\rfloor$ and assume without loss of generality that $\|\eta_1\|>\kappa(r, R)^{-\frac{1+\delta}{3}}$. Then
\begin{align}
  \label{eq:48}
  \begin{split}
  |\mathfrak m_R^{(r)}(\eta)|&\le
  \frac{1}{|B_{R}^{(r)}\cap\ZZ^r|}\sum_{x\in
  B_{R}^{(r)}\cap\ZZ^r}\frac{1}{M}\Big|\sum_{s=1}^Me^{2\pi i
  (x+se_1)\cdot\eta}\Big|\\
&+\frac{1}{M}\sum_{s=1}^M\frac{1}{|B_{R}^{(r)}\cap\ZZ^r|}\Big|\sum_{x\in
  B_{R}^{(r)}\cap\ZZ^r}e^{2\pi i
  x\cdot\eta}-e^{2\pi i
  (x+se_1)\cdot\eta}\Big|.
  \end{split}
\end{align}
Since $\kappa(r,R)\ge 1$ we now see that
\begin{align}
  \label{eq:49}
\frac{1}{M}\Big|\sum_{s=1}^Me^{2\pi i
  (x+se_1)\cdot\eta}\Big|\le M^{-1}\|\eta_1\|^{-1}\le
2\kappa(r, R)^{-\frac{1}{3}+\frac{2\delta}{3}}.  
\end{align}
We have assumed that $r\le R^{\delta}$, thus by Lemma \ref{lem:11}, with $z=se_1$
 and $s\le M\le \kappa(r, R)^{\frac{2-\delta}{3}}$, we get
 \begin{align}
   \label{eq:50}
   \begin{split}
   \frac{1}{|B_{R}^{(r)}\cap\ZZ^r|}\Big|\sum_{x\in
  B_{R}^{(r)}\cap\ZZ^r}e^{2\pi i
  x\cdot\eta}-e^{2\pi i
  (x+se_1)\cdot\eta}\Big|&\le\frac{1}{|B_{R}^{(r)}\cap\ZZ^r|}\big|\big(B_R^{(r)}\cap\ZZ^r\big)\triangle\big((se_1+B_R^{(r)})\cap\ZZ^r\big)\big|\\
                         &\le  8e\big(srR^{-1}e^{srR^{-1}}+e^{srR^{-1}}R^{-1+3\delta/2}\big)\\
                         &\le 16e^2\kappa(r, R)^{-\frac{1}{3}+\frac{2\delta}{3}},
   \end{split}
 \end{align}
 since $srR^{-1}\le \kappa(r, R)^{\frac{2-\delta}{3}}R^{-1+\delta}\le\kappa(r, R)^{-\frac{1}{3}+\frac{2\delta}{3}}\le1$, and $R^{-1+3\delta/2}\le R^{-\frac{1}{3}+\frac{2\delta}{3}}$, we have, for $R>16,$ 
 \[
|B_{R}^{(r)}\cap\ZZ^r|\ge|B_{R-r^{1/2}}^{(r)}|=|B_R^{(r)}|\bigg(1-\frac{r^{1/2}}{R}\bigg)^r\ge
|B_R^{(r)}|\big(1-r^{3/2}R^{-1}\big)\ge|B_R^{(r)}|/2.
\]
Combining \eqref{eq:48} with \eqref{eq:49} and \eqref{eq:50} we obtain
\[
|\mathfrak m_R^{(r)}(\eta)|\le (16e^2+2)\kappa(r, R)^{-\frac{1}{3}+\frac{2\delta}{3}}.
\]

Secondly, we assume  that
$\max\{\|\eta_1\|,\ldots,\|\eta_r\|\}\le\kappa(r, R)^{-\frac{1+\delta}{3}}$.  Observe that
by \eqref{eq:40} we have
\begin{align*}
  \bigg| \frac{1}{|B_{R}^{(r)}\cap\ZZ^r|}-
  \frac{1}{|B_{R}^{(r)}|}\bigg|\le\frac{eR^{-1+3\delta/2}}{|B_{R}^{(r)}\cap\ZZ^r|}
  \le\frac{2e\kappa(r, R)^{-\frac{1}{3}+\frac{2\delta}{3}}}{|B_{R}^{(r)}|}.
\end{align*}
Then
\begin{align}
\label{eq:42}
  \begin{split}
  |\mathfrak m_R^{(r)}(\eta)|&\le\Big|\mathfrak
  m_R^{(r)}(\eta)-\frac{1}{|B_{R}^{(r)}|}\mathcal
                               F(\ind{B_R^{(r)}})(\eta)\Big|+\frac{1}{|B_{R}^{(r)}|}\big|\mathcal
                               F(\ind{B_R^{(r)}})(\eta)\big|\\
  &\le 2e\kappa(r, R)^{-\frac{1}{3}+\frac{2\delta}{3}}+\frac{1}{|B_{R}^{(r)}\cap\ZZ^r|}\Big|\sum_{x\in
  B_{R}^{(r)}\cap\ZZ^r}e^{2\pi i
  x\cdot\eta}-\int_{B_R^{(r)}}e^{2\pi i
  y\cdot\eta}\dif y\Big|+\frac{1}{|B_{R}^{(r)}|}\big|\mathcal
                               F(\ind{B_R^{(r)}})(\eta)\big|.
  \end{split}
\end{align}
Let $Q^{(r)}=[-1/2, 1/2]^r$ and note that by Lemma \ref{lem:11}, with
$z=t\in Q^{(r)}$ we get
\begin{align}
  \label{eq:51}
  \begin{split}
  \frac{1}{|B_{R}^{(r)}\cap\ZZ^r|}\Big|\sum_{x\in
  B_{R}^{(r)}\cap\ZZ^r}e^{2\pi i
  x\cdot\eta}&-\int_{B_R^{(r)}}e^{2\pi i
  y\cdot\eta}\dif y\Big|\\
&=\frac{1}{|B_{R}^{(r)}\cap\ZZ^r|}\Big|\sum_{x\in
  \ZZ^r}\int_{Q^{(r)}}e^{2\pi i
  x\cdot\eta}\ind{B_R^{(r)}}(x)-e^{2\pi i
  (x+t)\cdot\eta}\ind{B_R^{(r)}}(x+t)\dif t\Big|\\
&\le\frac{1}{|B_{R}^{(r)}\cap\ZZ^r|}\int_{Q^{(r)}}\big|(B_R^{(r)}\cap\ZZ^r)\triangle\big((t+B_R^{(r)})\cap\ZZ^r\big)\big|\dif
t\\
&+
  \frac{1}{|B_{R}^{(r)}\cap\ZZ^r|}\sum_{x\in
  \ZZ^r}\ind{B_R^{(r)}}(x)\int_{Q^{(r)}}|e^{2\pi i
  x\cdot\eta}-e^{2\pi i
  (x+t)\cdot\eta}|\dif t\\
& \le16e^2\kappa(r, R)^{-\frac{1}{3}+\frac{2\delta}{3}}+2\pi\big(\|\eta_1\|+\ldots+\|\eta_r\|\big)\\
&\le16e^2\kappa(r, R)^{-\frac{1}{3}+\frac{2\delta}{3}}+2\pi r\kappa(r, R)^{-\frac{1+\delta}{3}}.
  \end{split}
\end{align}
Finally, by Lemma \ref{lem:19}  we obtain
\begin{align}
  \label{eq:53}
  \begin{split}
  \frac{1}{|B_R^{(r)}|}|\mathcal F(\ind{B_R^{(r)}})(\eta)|\le  C\big(\kappa(r, R)\|\eta\|\big)^{-1}.
  \end{split}
\end{align}
Combining \eqref{eq:42} with \eqref{eq:51} and \eqref{eq:53} we obtain
 that
\[
|\mathfrak m_R^{(r)}(\eta)|\le(16e^2+2e)\kappa(r, R)^{-\frac{1}{3}+\frac{2\delta}{3}}+2\pi r\kappa(r, R)^{-\frac{1+\delta}{3}}
+C\big(\kappa(r, R)\|\eta\|\big)^{-1},
\]
which completes the proof.
\end{proof}

\begin{lemma}
\label{lem:13}
For every $\delta\in(0, 1/2)$ and $\varepsilon\in(0, 1/50]$ there is a
constant $C_{\delta, \varepsilon}>0$ such
that for every $d, N\in\NN$, if $r$ is an integer such that $1\le r\le d$ and 
$\max\{1, \varepsilon^{\frac{3\delta}{2}}\kappa(d, N)^{\delta}/2\}\le r\le \max\{1,\varepsilon^{\frac{3\delta}{2}}\kappa(d, N)^{\delta}\}$,
then for every $\xi=(\xi_1,\ldots,\xi_d)\in\TT^d$ we have
\begin{align*}
  |\mathfrak m_N(\xi)|\le C_{\delta, \varepsilon}\big(\kappa(d, N)^{-\frac{1}{3}+\frac{2\delta}{3}}+(\kappa(d, N)\|\eta\|)^{-1}\big),
\end{align*}
 where $\eta=(\xi_1,\ldots, \xi_r)$.
\end{lemma}
\begin{proof}
If $\kappa(d, N)\le \varepsilon^{-\frac{3}{2}}$, then there is nothing to do, since the implied constant in question is allowed to depend on $\delta$ and $\varepsilon$. We will assume that $\kappa(d, N)\ge \varepsilon^{-\frac{3}{2}}$, which ensures that $\kappa(d, N)\ge10$.
In view of Lemma \ref{lem:9} we have 
\begin{align}
\label{eq:35}
  |\mathfrak m_N(\xi)|\le\sup_{R\ge \varepsilon^{3/2}\kappa(d, N) r^{1/2}}|\mathfrak
  m_{R}^{(r)}(\eta)|+4e^{-\frac{\varepsilon r}{10}},
\end{align}
where $\eta=(\xi_1,\ldots,\xi_r)$. By Lemma \ref{lem:20}, since $r\le \varepsilon^{\frac{3\delta}{2}}\kappa(d, N)^{\delta}\le \kappa(r, R)^{\delta}\le R^{\delta}$, we obtain
\begin{align}
\label{eq:36}
|\mathfrak
  m_{R}^{(r)}(\eta)|\lesssim \kappa(r, R)^{-\frac{1}{3}+\frac{2\delta}{3}}
+r\kappa(r, R)^{-\frac{1+\delta}{3}}+\big(\kappa(r, R)\|\eta\|\big)^{-1}.
\end{align}
Combining \eqref{eq:35} and \eqref{eq:36} with our assumptions we obtain the desired claim. 
\end{proof}

\subsection{All together} We have prepared all necessary tools to
prove inequality \eqref{eq:23}. We shall be working under the
assumptions of Lemma \ref{lem:13} with $\delta=2/7$.

\begin{proof}[Proof of Proposition \ref{prop:2}]
Assume that $\varepsilon=1/50$. If
$\kappa(d, N)\le2\cdot50^{\frac{3}{7}}$ then clearly \eqref{eq:23}
holds. Therefore, we can assume that $d, N\in\NN$ satisfy
$\kappa(d, N)\ge2\cdot50^{\frac{3}{7}}$.  We choose an integer
$1\le r\le d$ satisfying
$50^{-\frac{3}{7}}\kappa(d, N)^{\frac{2}{7}}/2\le r\le 50^{-\frac{3}{7}}\kappa(d, N)^{\frac{2}{7}}$, it is possible 
since $50^{-\frac{3}{7}}\kappa(d, N)^{\frac{2}{7}}\le 50^{-\frac{1}{7}}d^{\frac{1}{7}}\le 2d/3$.
We will also assume that $\|\xi_1\|\ge\ldots\ge\|\xi_d\|$ and we shall
distinguish two cases.  Suppose first that
\begin{align}
  \label{eq:69}
  \|\xi_1\|^2+\ldots+\|\xi_r\|^2\ge\frac{1}{4}\|\xi\|^2.
\end{align}
Then in view of Lemma \ref{lem:13} (with $\delta=2/7$ and $r\simeq\kappa(d, N)^{\frac{2}{7}}$)  and \eqref{eq:69} we obtain 
\begin{align*}
  |\mathfrak m_N(\xi)|\le C\big(\kappa(d, N)^{-\frac{1}{7}}+(\kappa(d, N)\|\xi\|)^{-1}\big),
\end{align*}
and we are done. So we can assume that
\begin{align}
  \label{eq:54}
  \|\xi_1\|^2+\ldots+\|\xi_r\|^2\le\frac{1}{4}\|\xi\|^2.  
\end{align}
Let $\varepsilon_1=1/10$
and assume first that
\begin{align}
  \label{eq:55}
  \|\xi_j\|\le\frac{\varepsilon_1^{1/2}}{10\kappa(d,N)}\quad\text{ for all } \quad r\le
  j\le d.
\end{align}
Using \eqref{eq:43} and the Cauchy--Schwarz inequality we have
\begin{align}
  \label{eq:56}
  \begin{split}
  |\mathfrak m_N(\xi)|^2&\le \frac{1}{|B_N\cap\ZZ^d|}\sum_{x\in
                          B_N\cap\ZZ^d}\prod_{j=1}^d \cos^2(2\pi x_j \xi_j)\\
  &\le \frac{1}{|B_N\cap\ZZ^d|}\sum_{x\in
    B_N\cap\ZZ^d}\prod_{j=1}^d (1-\sin^2(2\pi x_j \xi_j))\\
  &\le \frac{1}{|B_N\cap\ZZ^d|}\sum_{x\in
  B_N\cap\ZZ^d}\exp\Big(-\sum_{j=r+1}^d\sin^2(2\pi x_j \xi_j)\Big).
  \end{split}
\end{align}
For $x\in B_N\cap\ZZ^d$ we define
\begin{align*}
  I_x&=\{i\in\NN_d\colon \varepsilon\kappa(d, N)\le |x_i|\le 2\varepsilon_1^{-1/2}\kappa(d,N) \},\\
  I_x'&=\{i\in\NN_d\colon 2\varepsilon_1^{-1/2}\kappa(d,N)< |x_i| \},\\
    I_x''&=\{i\in\NN_d\colon \varepsilon\kappa(d,N)\le |x_i|\}=I_x\cup I_x'.
\end{align*}
 and
\[
E=\big\{x\in B_N\cap\ZZ^d\colon |I_x|\ge\varepsilon_1 d/2\big\}.
\]
Observe that
\begin{align*}
  E^{\bf c}=&\big\{x\in B_N\cap\ZZ^d\colon |I_x|<\varepsilon_1 d/2\big\}
  =\big\{x\in B_N\cap\ZZ^d\colon |I_x''|<\varepsilon_1 d/2+|I_x'|\big\}\\
   \subseteq&\big\{x\in B_N\cap\ZZ^d\colon |I_x''|<\varepsilon_1 d/2+|I_x'|\text{
     and } |I_x'|\le \varepsilon_1 d/2\big\}
  \cup
  \big\{x\in B_N\cap\ZZ^d\colon |I_x'|> \varepsilon_1 d/2\big\}.
\end{align*}
Then it is not difficult to see that
\[
E^{\bf c}\subseteq \big\{x\in B_N\cap\ZZ^d\colon |I_x''|<\varepsilon_1 d\big\},
\]
since $  \big\{x\in B_N\cap\ZZ^d\colon |I_x'|> \varepsilon_1 d/2\big\}=\emptyset$.
Then by Lemma \ref{lem:5} with $\varepsilon_2=\varepsilon$,
we obtain
\begin{align*}
  |E^{\bf c}|\le |\big\{x\in B_N\cap\ZZ^d\colon |I_x''|<\varepsilon_1 d\big\}|\le 2e^{-\frac{d}{10}}|B_N\cap\ZZ^d|.
\end{align*}
Therefore, by \eqref{eq:56} we have
\begin{align}
\label{eq:57}
  \begin{split}
  |\mathfrak m_N(\xi)|^2
  &\le \frac{1}{|B_N\cap\ZZ^d|}\sum_{x\in
  B_N\cap\ZZ^d}\exp\Big(-\sum_{j\in I_x\cap J_r}\sin^2(2\pi x_j
\xi_j)\Big)\ind{E}(x)+2e^{-\frac{d}{10}},
\end{split}
\end{align}
where $J_r=\{r+1,\ldots, d\}$. Using \eqref{eq:103} and definition $I_x$ we have
\[
\sin^2(2\pi x_j \xi_j)\ge 16|x_j|^2\|\xi_j\|^2\ge 16\varepsilon^2\kappa(d,N)^2\|\xi_j\|^2,
\]
since $2|x_j|\|\xi_j\|\le 1/2$ by \eqref{eq:55}, and consequently we estimate \eqref{eq:57} and obtain for some $C, c>0$ that
\begin{align}
  \label{eq:58}
  \begin{split}
  \frac{1}{|B_N\cap\ZZ^d|}&\sum_{x\in
  B_N\cap\ZZ^d}\exp\Big(-\sum_{j\in I_x\cap J_r}\sin^2(2\pi x_j
  \xi_j)\Big)\ind{E}(x)\\
  &\le
  \frac{1}{|B_N\cap\ZZ^d|}\sum_{x\in
  B_N\cap\ZZ^d\cap E}\exp\Big(-16\varepsilon^2\kappa(d,N)^2\sum_{j\in I_x\cap
  J_r}\|\xi_j\|^2\Big)\le Ce^{-c\kappa(d,N)^2\|\xi\|^2}.
  \end{split}
\end{align}
In order to obtain the last inequality in \eqref{eq:58} observe that 
\begin{align*}
\frac{1}{|B_N\cap\ZZ^d|}&\sum_{x\in B_N\cap\ZZ^d\cap E}
\exp\Big(-16\varepsilon^2\kappa(d,N)^2\sum_{j\in I_x\cap J_r}\|\xi_j\|^2\Big)\\
&=\frac{1}{|B_N\cap\ZZ^d|}\sum_{x\in
  B_N\cap\ZZ^d\cap E}\frac{1}{d!}\sum_{\sigma\in{\rm Sym}(d)}\exp\Big(-16\varepsilon^2\kappa(d,N)^2\sum_{j\in \sigma(I_x)\cap
  J_r}\|\xi_j\|^2\Big)\\
  &=\frac{1}{|B_N\cap\ZZ^d|}\sum_{x\in
  B_N\cap\ZZ^d\cap E}\mathbb E\bigg[\exp\Big(-16\varepsilon^2\kappa(d,N)^2\sum_{j\in \sigma(I_x)\cap
  J_r}\|\xi_j\|^2\Big)\bigg],
\end{align*}
since $\sigma\cdot (B_N\cap\ZZ^d\cap E)=B_N\cap\ZZ^d\cap E$ for every $\sigma\in{\rm Sym}(d)$. 
Appealing now to  Lemma \ref{lem:7} with $\delta_1=\varepsilon_1/2$, $d_0=r$, $I=I_x$ and $\delta_0=3/5$, we conclude that
\begin{align*}
\mathbb E\bigg[\exp\Big(-16\varepsilon^2\kappa(d,N)^2\sum_{j\in \sigma(I_x)\cap
  J_r}\|\xi_j\|^2\Big)\bigg]\le C\exp\Big(-c'\kappa(d,N)^2\sum_{j=r+1}^d\|\xi_j\|^2\Big),
  \end{align*}
  for some $c'>0$ and for all $x\in B_N\cap\ZZ^d\cap E$.
  This proves \eqref{eq:58} since by
\eqref{eq:54} we obtain
\begin{align*}
\exp\Big(-c'\kappa(d,N)^2\sum_{j=r+1}^d\|\xi_j\|^2\Big)\le \exp\Big(-\frac{c'\kappa(d,N)^2}{4}\sum_{j=1}^d\|\xi_j\|^2\Big).
\end{align*}

Assume now that \eqref{eq:55} does not hold. Then
\begin{align}
\label{eq:44}
  \|\xi_j\|\ge\frac{\varepsilon_1^{1/2}}{10\kappa(d,N)}\quad\text{ for all
  } \quad 1\le j\le r.
\end{align}
Hence \eqref{eq:44} gives that
\begin{align*}
  \|\xi_1\|^2+\ldots+\|\xi_r\|^2\ge\frac{\varepsilon_1r}{100\kappa(d,N)^2}.
\end{align*}
Therefore, we invoke Lemma \ref{lem:13} with $\eta=(\xi_1,\ldots, \xi_r)$ again and obtain 
\begin{align*}
  |\mathfrak m_N(\xi)|&\lesssim\kappa(d, N)^{-\frac{1}{7}}+(\kappa(d, N)\|\eta\|)^{-1}\\
  &\lesssim \kappa(d, N)^{-\frac{1}{7}},
\end{align*}
since $r\simeq\kappa(d, N)^{\frac{2}{7}}$. This completes the proof of Proposition \ref{prop:2}.
\end{proof}

\section{Estimates for the dyadic maximal function: small
  scales}
\label{sec:4}
This section is intended to provide bounds independent of the
dimension for the dyadic maximal function with supremum taken over all
dyadic numbers $N$ such that $1\lesssim N\lesssim d^{1/2}$. Theorem \ref{thm:2} combined with Theorem \ref{thm:1} from the previous section implies our main result Theorem \ref{thm:0}.

\begin{theorem}
\label{thm:2}
Let $C_0>0$ and define $\mathbb D_{C_0}=\{N\in\mathbb D:1\le N\le C_0d^{1/2}\}$. Then there exists a constant $C>0$ independent of
dimension such that for every $f\in \ell^2(\ZZ^d)$ we have
\begin{align}
\label{eq:106}
  \big\|\sup_{N\in \mathbb D_{C_0}}\mathcal M_Nf\big\|_{\ell^2}\le C\|f\|_{\ell^2}.
\end{align}
\end{theorem}

The strategy of the  proof of Theorem \ref{thm:2} is much the same as for the proof of Theorem \ref{thm:1}. However, the approximating multiplies are different and they depend on the size of the set
\begin{align}
\label{eq:107}
  V_{\xi}=\{i\in\NN_d\colon \cos(2\pi\xi_i)<0\}=\{i\in\NN_d\colon 1/4<|\xi_i|\le1/2\} \quad \text{for} \quad \xi\in\TT^d.
\end{align}
We will approximate the maximal function from \eqref{eq:106} by the maximal functions associated with the following multipliers
\begin{align}
  \label{eq:96}
  \lambda^1_N(\xi)&=  e^{-\kappa(d, N)^2\sum_{i=1}^d\sin^2(\pi\xi_i)} && \text{ if
                                                     } |V_{\xi}|\le d/2,\\
 \label{eq:97}\lambda^2_N(\xi)&=\frac{1}{|B_N\cap\ZZ^d|}\Big(\sum_{x\in B_N\cap\ZZ^d}(-1)^{\sum_{i=1}^dx_i}\Big)
  e^{-\kappa(d, N)^2\sum_{i=1}^d\cos^2(\pi\xi_i)} && \text{ if
                                                     } |V_{\xi}|\ge d/2.
\end{align}
In Proposition \ref{prop:4}, which is the main results of this section, we show that multiplies \eqref{eq:96} and \eqref{eq:97} are close to the multiplier $\mathfrak m_N$ defined in \eqref{eq:116} in the sense of inequalities \eqref{eq:94} and \eqref{eq:95} respectively.

\begin{proposition}
\label{prop:4}
For every  $d,N\in\NN$, if  $N\ge2^{9/2}$ and $\kappa(d, N)\le 1/5$, then for every $\xi\in\TT^d$ we have the following bounds with the constant $c\in(0, 1)$ as in \eqref{eq:70}. Namely,
\begin{enumerate}[label*={\arabic*}.]
\item if $|V_{\xi}|\le d/2$, then
\begin{align}
  \label{eq:94}
  |\mathfrak m_N(\xi)-\lambda_N^1(\xi)|
  \le 17\min\Big\{e^{-\frac{c\kappa(d, N)^2}{400}\sum_{i=1}^d\sin^2(\pi\xi_i)},
  \kappa(d, N)^2\sum_{i=1}^d\sin^2(\pi\xi_i)\Big\},
\end{align}
\item if $|V_{\xi}|\ge d/2$, then
\begin{align}
  \label{eq:95}
  |\mathfrak m_N(\xi)-\lambda_N^2(\xi)|
\le 17\min\Big\{e^{-\frac{c\kappa(d, N)^2}{400}\sum_{i=1}^d\cos^2(\pi\xi_i)}, \kappa(d, N)^2\sum_{i=1}^d\cos^2(\pi\xi_i)\Big\}.
\end{align}
\end{enumerate}
\end{proposition}

Throughout this section all of the estimates will be also described in terms of the proportionality constant $\kappa(d, N)$ from \eqref{eq:46}.
Assume momentarily that Proposition \ref{prop:4} has been proven and let us deduce Theorem \ref{thm:2}.

\begin{proof}[Proof of Theorem \ref{thm:2}]
Let $f\in\ell^2(\ZZ^d)$ and we write $f=f_1+f_2$, where
$\hat{f}_1(\xi)=\hat{f}(\xi)\ind{\{\eta\in\TT^d\colon |V_{\eta}|\le d/2\}}(\xi)$. Then
\begin{align*}
  \big\|\sup_{N\in\mathbb D_{C_0}}|\mathcal F^{-1}(\mathfrak m_N\hat{f})|\big\|_{\ell^2}\le&
    \sum_{i=1}^2\big\|\sup_{N\in\mathbb D_{C_0}}|\mathcal F^{-1}(\mathfrak m_N\hat{f}_i)|\big\|_{\ell^2}\\
    \le& \sum_{i=1}^2\big\|\sup_{N\in\mathbb D_{C_0}}|\mathcal F^{-1}(\lambda_N^i\hat{f}_i)|\big\|_{\ell^2}
    +\sum_{i=1}^2\Big\|\Big(\sum_{N\in\mathbb D_{C_0}}\big|\mathcal F^{-1}\big((\mathfrak m_N-\lambda_N^i)\hat{f}_i\big)\big|^2\Big)\Big\|_{\ell^2}.
\end{align*}
The usual square function argument permits therefore to reduce the problem to controlling the maximal functions associated with the
multipliers $\lambda_N^1$ and  $\lambda_N^2$. Since by Proposition \ref{prop:4} we have 
\[
\sum_{i=1}^2\Big\|\Big(\sum_{N\in\mathbb D_{C_0}}\big|\mathcal F^{-1}\big((\mathfrak m_N-\lambda_N^i)\hat{f}_i\big)\big|^2\Big)\Big\|_{\ell^2}\lesssim \|f\|_{\ell^2}.
\]
We only have to bound the maximal functions. Observe that
\begin{align}
  \label{eq:99}
  \big\|\sup_{N\in\mathbb D_{C_0}}|\mathcal F^{-1}(\lambda_N^1\hat{f}_1)|\big\|_{\ell^2}\lesssim \|f\|_{\ell^2},
\end{align}
since the multiplier $\lambda_N^1$ corresponds to the semi-group of
contractions $P_t$, which was defined in the proof of Theorem \ref{thm:1}. The multiplier $\mathfrak p_t$ corresponding to $P_t$ satisfies $\lambda_N^1(\xi)=\mathfrak p_{\kappa(d, N)^2}(\xi)$. Therefore, by \eqref{eq:47} we obtain \eqref{eq:99} as desired.
It is also  not difficult to see that
\begin{align}
\label{eq:100}
  \big\|\sup_{N\in\mathbb D_{C_0}}|\mathcal F^{-1}(\lambda_N^2\hat{f}_2)|\big\|_{\ell^2}\lesssim \|f\|_{\ell^2}.
\end{align}
In fact \eqref{eq:100} can be deduced from \eqref{eq:99}. For this
purpose we denote $F_2(x)=(-1)^{\sum_{j=1}^dx_j}f_2(x)$, hence
\[
\hat{F}_2(\xi)=\sum_{x\in\ZZ^d}e^{2\pi i x\cdot\xi}(-1)^{\sum_{j=1}^dx_j}f_2(x)=\sum_{x\in\ZZ^d}e^{2\pi i x\cdot(\xi+1/2)}f_2(x)=\hat{f}_2(\xi+1/2).
\]
Thus we write
\begin{align*}
  \big\|\sup_{N\in\mathbb D_{C_0}}|\mathcal F^{-1}(\lambda_N^2\hat{f}_2)|\big\|_{\ell^2}
  &\le \Big\|\sup_{N\in\mathbb D_{C_0}}\Big|\int_{\TT^d}e^{-\kappa(d, N)^2\sum_{j=1}^d\cos^2(\pi\xi_j)}\hat{f}_2(\xi)e^{-2\pi i x\cdot\xi}\dif \xi\Big|\Big\|_{\ell^2}\\
  &\le \Big\|\sup_{N\in\mathbb D_{C_0}}\Big|\int_{\TT^d}e^{-\kappa(d, N)^2\sum_{j=1}^d\sin^2(\pi\xi_j)}\hat{F}_2(\xi)e^{-2\pi i x\cdot\xi}\dif \xi\Big|\Big\|_{\ell^2}\\
  &\le \|f\|_{\ell^2},
\end{align*}
since $\|F_2\|_{\ell^2}=\|f_2\|_{\ell^2}\le \|f\|_{\ell^2}$.
\end{proof}

\subsection{Some preparatory estimates}
We begin with a very useful lemma, which will allow us to control efficiently sizes of certain error sets, which say, to some extent, that a large amount of mass of $B_N\cap\ZZ^d$ is concentrated on the set $\{-1, 0, 1\}^d$. 
\begin{lemma}
\label{lem:15}
For every $d, N\in\NN$, if $n=N^2$ and  $\kappa(d, N)\le 1/5$ and  $n\ge k\ge2^9\max\{1, \kappa(d, N)^6n\}$, then
\begin{align}
  \label{eq:83}
  |\{x\in B_N\cap\ZZ^d\colon |\{i\in\NN_d\colon x_i=\pm1\}|\le n-k\}|\le 2^{-k+1}|B_N\cap\ZZ^d|.
\end{align}
In particular, we have
\begin{align}
  \label{eq:76}
  |\{x\in B_N\cap\ZZ^d\colon \sum_{\substack{i=1\\
  |x_i|\ge2}}^dx_i^2>k\}|\le 2^{-k+1}|B_N\cap\ZZ^d|.
\end{align}
\end{lemma}
\begin{proof}
If $n=N^2$ then $nd^{-1}=\kappa(d, N)^2$. It is easy to  see that \eqref{eq:76} follows from \eqref{eq:83}, since 
\begin{align*}
  |\{x\in B_N\cap\ZZ^d\colon \sum_{\substack{i=1\\
  |x_i|\ge2}}^dx_i^2>k\}|&\le|\{x\in B_N\cap\ZZ^d\colon \sum_{\substack{i=1\\
  |x_i|=1}}^dx_i^2\le n-k\}|\\
  &=|\{x\in B_N\cap\ZZ^d\colon |\{i\in\NN_d\colon x_i=\pm1\}|\le n-k\}|.
\end{align*}
To prove \eqref{eq:83} let us define for every $m\in\NN_n$ the set
\[
E_m=\{x\in B_N\cap\ZZ^d\colon \sum_{\substack{i=1\\ |x_i|=1}}^dx_i^2=n-m\},
\]
and observe that
\begin{align*}
|\{x\in B_N\cap\ZZ^d\colon \sum_{\substack{i=1\\
  |x_i|=1}}^dx_i^2\le n-k\}|
  \le\sum_{m=k}^{n}|E_m|.
\end{align*}
In view of this bound it suffices to  show, for all
$m\in\NN_n\setminus\NN_{k-1}$, that
\begin{align}
  \label{eq:77}
|E_m|\le 2^{-m}|B_N\cap\ZZ^d|.
\end{align}
Our task now is to prove \eqref{eq:77}. If $x\in E_m$ then
\begin{align*}
  \quad \sum_{\substack{i=1\\
  |x_i|\ge2}}^dx_i^2\le m
  \end{align*}
and consequently we obtain that
\begin{align}
  \label{eq:79}
|\{i\in\NN_d\colon |x_i|\ge2\}|\le \lfloor m/4\rfloor.  
\end{align} 
We now establish the following upper bound for the sets $E_m$
\begin{align}
  \label{eq:78}
|E_m|
  \le2^{n-m}\binom{d}{n-m}\binom{d-n+m}{\lfloor m/4\rfloor}
  \big|B_{\sqrt{m}}^{(\lfloor m/4\rfloor)}\cap\ZZ^{\lfloor m/4\rfloor}\big|.
\end{align}
If $x=(x_1,\ldots, x_d)\in E_m$ then precisely $n-m$ of its coordinates are $\pm1$.
Due to \eqref{eq:79} at most $\lfloor m/4\rfloor$ coordinates of $x$ have
an absolute value at least $2$. Finally, the remaining
$d-n+m-\lfloor m/4\rfloor$ coordinates vanish. This allows us
to justify \eqref{eq:78}.  There are
$2^{n-m}$ sequences of length $n-m$ whose elements are $\pm1$, the
elements of these sequences can be placed into $\binom{d}{n-m}$ spots.
We have also at most
$\big|B_{\sqrt{m}}^{(\lfloor m/4\rfloor)}\cap\ZZ^{\lfloor m/4\rfloor}\big|$
sequences with elements whose absolute value is at least $2$, these
elements can be placed into $\binom{d-n+m}{\lfloor m/4\rfloor}$
spots. Finally there is only one way to put zeros into the
remaining $d-n+m-\lfloor m/4\rfloor$ spots.

By Lemma \ref{lem:4} we get for $m\ge8$ that
\begin{align}
  \label{eq:80}
  \begin{split}
  \big|B_{\sqrt{m}}^{(\lfloor m/4\rfloor)}\cap\ZZ^{\lfloor m/4\rfloor}\big|
  &\le (2\pi e)^{m/8}\bigg(\frac{m}{m/4-1}+\frac{1}{4}\bigg)^{m/8}\\
  &\le (2\pi e)^{m/8}\bigg(\frac{4m}{m-4}+\frac{1}{4}\bigg)^{m/8}\\
  &\le (17\pi e)^{m/8}.
  \end{split}
  \end{align}
Moreover, we obtain 
\begin{align}
\label{eq:52}
\begin{split}
\binom{d}{n-m}\binom{d-n+m}{\lfloor m/4\rfloor}\binom{d}{n}^{-1}&=
\frac{n!(d-n)!}{(n-m)!(d-n+m-\lfloor m/4\rfloor)!\lfloor m/4\rfloor!}\\
&\le\frac{n^{\lfloor m/4\rfloor}}{\lfloor m/4\rfloor!}\bigg(\frac{n}{d-n}\bigg)^{m-\lfloor m/4\rfloor}\\
&\le\bigg(\frac{en}{\lfloor m/4\rfloor}\bigg)^{\lfloor m/4\rfloor}\bigg(\frac{n}{d-n}\bigg)^{m-\lfloor m/4\rfloor}\\
&\le\bigg(\frac{4en}{m}\bigg)^{m/4}\bigg(\frac{\kappa(d, N)^2}{1-\kappa(d, N)^2}\bigg)^{3m/4},
\end{split}
\end{align}
since $\lfloor m/4\rfloor!\ge(\lfloor m/4\rfloor/e)^{\lfloor m/4\rfloor}$ and for every $a, b>0$ the function $(0, a/e]\ni t\mapsto \big(\frac{a}{t}\big)^{bt}$ is increasing.

We also have
\begin{align}
  \label{eq:82}
  2^n\binom{d}{n}\le |B_N\cap\ZZ^d|.
\end{align}
Thus taking into account \eqref{eq:78}, \eqref{eq:80}, \eqref{eq:52} and
\eqref{eq:82}, we obtain
\begin{align*}
|E_m|&\le2^n\binom{d}{n}2^{-m}(17\pi e)^{m/8}\bigg(\frac{4en}{m}\bigg)^{m/4}\bigg(\frac{\kappa(d, N)^2}{1-\kappa(d, N)^2}\bigg)^{3m/4}\\
  &\le \bigg(\frac{272\pi e^3}{256}\bigg)^{m/8}\bigg(\frac{n}{m}\bigg)^{m/4}\bigg(\frac{\kappa(d, N)^2}{1-\kappa(d, N)^2}\bigg)^{3m/4}|B_N\cap\ZZ^d|\\
  &\le 2^m\bigg(\frac{n}{m}\bigg)^{m/4}\bigg(\frac{\kappa(d, N)^2}{1-\kappa(d, N)^2}\bigg)^{3m/4}|B_N\cap\ZZ^d|\\
  &\le 2^m\bigg(\frac{25}{24}\bigg)^{3m/4}\bigg(\frac{n}{m}\bigg)^{m/4}\kappa(d, N)^{3m/2}|B_N\cap\ZZ^d|\\
  &\le 2^{-m}|B_N\cap\ZZ^d|,
\end{align*}
since $m\ge2^{9}\kappa(d, N)^6n$. The proof now is completed. 
\end{proof}

\subsection{Analysis exploiting the Krawtchouk polynomials}

The proof of Proposition \ref{prop:5} will rely  on the properties of the Krawtchouk polynomials. We need to
introduce some definitions and formulate basic facts.
For every $n\in\NN_0$ and integers  $x, k\in[0, n]$ we define the
$k$-th Krawtchouk polynomial
\begin{align}
  \label{eq:68}
  \Bbbk_k^{(n)}(x)=\frac{1}{\binom{n}{k}}\sum_{j=0}^k(-1)^j\binom{x}{j}\binom{n-x}{k-j}.
\end{align}
We gather some facts about the Krawtchouk polynomials in the theorem stated below.
\begin{theorem}
\label{thm:100}
For every $n\in\NN_0$ and integers  $x, k\in[0, n]$ we have
\begin{enumerate}[label*={\arabic*}.]

\item \label{item:1}{\it Symmetry:} $\Bbbk_k^{(n)}(x)=\Bbbk_x^{(n)}(k)$.

\item \label{item:2}{{\it Reflection symmetry:}}
$\Bbbk_k^{(n)}(n-x)=(-1)^k\Bbbk_k^{(n)}(x)$.

\item \label{item:3}{{\it Orthogonality:}} for every $k, m\in[0, n]$
\[
\sum_{x=0}^n\binom{n}{x}\Bbbk_k^{(n)}(x)\Bbbk_m^{(n)}(x)=2^n\binom{n}{k}^{-1}\delta_k(m).
\]

\item \label{item:4}{\it Roots:} the roots of  $\Bbbk_k^{(n)}(x)$ are real, distinct, and lie in
$\big[\frac{n}{2}-\sqrt{k(n-k)}, \frac{n}{2}+\sqrt{k(n-k)}\:\big]$.
\item \label{item:5}{\it A uniform bound:} there exists a constant $c\in(0, 1)$ such that
for all $n\in\NN_0$ the following inequality
\begin{align}
  \label{eq:70}
  \big|\Bbbk_k^{(n)}(x)\big|\le e^{-\frac{ckx}{n}}
\end{align}
holds for all integers $0\le x, k\le n/2$.

\end{enumerate}
\end{theorem}
The proof of Theorem \ref{thm:100} can be found in \cite{HKS}, see also the references therein. In Proposition \ref{prop:5}, using Theorem \ref{thm:100}, we will be able to describe the decay at infinity of the multipliers corresponding to our averages. In fact, we will only use properties \ref{item:1}, \ref{item:2} and \ref{item:5}.  Properties \ref{item:3} and \ref{item:4}  are only provided to give an idea about the objects we are going to work with. The support of $x\in\RR^d$ will be denoted by ${\rm supp}\: x=\{i\in\NN_d: x_i\not=0\}$.

\begin{proposition}
	\label{prop:5}
	For every  $d, n\in\NN$, if $N\ge 2^{9/2}$ and $\kappa(d, N)\le 1/5$ then  for all $\xi\in\TT^d$ we have the following estimate
	\begin{align}
\label{eq:98}
|\mathfrak m_N(\xi)|\le 8e^{-\frac{c\kappa(d, N)^2}{100}\sum_{i=1}^d\sin^2(\pi\xi_i)}
+8e^{-\frac{c\kappa(d, N)^2}{100}\sum_{i=1}^d\cos^2(\pi\xi_i)},
	\end{align}
        where $c\in(0, 1)$ is the constant as in \eqref{eq:70}.
\end{proposition}

\begin{proof}
We let $n=N^2\ge2^9$ and for any $x\in\ZZ^d$ we define the set $I_x=\{i\in\NN_d\colon x_i=\pm1\}$ and  the set
$E=\{x\in B_N\cap\ZZ^d\colon |I_x|>n/2\}$. By Lemma \ref{lem:15}, with $k=\lfloor n/2\rfloor$, we
see that
\[
|E^{\bf c}|\le  2^{-n/2+2}|B_N\cap\ZZ^d|\le 2^{-\frac{\kappa(d, N)^2d}{2}+2}|B_N\cap\ZZ^d|\le4 e^{-\frac{\kappa(d, N)^2d}{4}}|B_N\cap\ZZ^d|.
\]
In view of this estimate it now suffices to show that
\begin{align}
\label{eq:59}
\frac{1}{|B_N\cap\ZZ^d|}\Big|\sum_{x\in B_N\cap \ZZ^d\cap E}
\prod_{j=1}^d \cos(2\pi x_j \xi_j)\Big|\le
6e^{-\frac{c\kappa(d, N)^2}{100}\sum_{i=1}^d\sin^2(\pi\xi_i)}
+6e^{-\frac{c\kappa(d, N)^2}{100}\sum_{i=1}^d\cos^2(\pi\xi_i)}.
\end{align}
For this purpose we decompose every $x\in B(N)=B_N\cap \ZZ^d\cap E=E$ uniquely as $x=Y(x)+Z(x),$ where $Z(x)=(Z_1(x),\ldots, Z_d(x))$ is given by
\begin{align*}
Z_j(x)=
\begin{cases}
x_j, &\textrm{ if } |x_j|\ge 2\\
0, &\textrm{ if } |x_j|\le 1,
\end{cases}
\end{align*}
and we shall exploit the following disjoint decomposition
\begin{align}
\label{eq:60}
B(N)=\bigcup_{z\in Z(B(N))}z+Y(N,z),
\end{align}
where
\begin{align*}
Y(N, z)=\big\{ y\in\{-1,0,1\}^d\colon  z+y\in B(N),\ \text{ and } \supp y\subseteq (\supp z)^{\bf c}\big\}.
\end{align*}
Note that then $|Y(N,z)|$ depends only on $N$ and $|\supp z|.$  We  abbreviate $\supp z$ to $S_z$ and using \eqref{eq:60} we write
\begin{align*}
\frac{1}{|B_N\cap\ZZ^d|}&\Big|\sum_{x\in B_N\cap \ZZ^d\cap E}
\prod_{j=1}^d \cos(2\pi x_j \xi_j)\Big|\\
&=\frac{1}{|B_N\cap\ZZ^d|}
\Big|\sum_{z\in Z(B(N))}\prod_{j \in S_z} \cos(2\pi z_j \xi_j) \sum_{y\in Y(N, z)}
\prod_{j\in \NN_d\setminus S_z} \cos(2\pi y_j \xi_j)\Big|.
\end{align*}
We claim that for every $z\in Z(B(N))$, whenever $Y(N,z)$ is non-empty, then 
\begin{align}
\label{eq:62}
\begin{split}
\frac{1}{|Y(N,z)|}\Big|&\sum_{y\in Y(N, z)}
\prod_{j\in \NN_d\setminus S_z} \cos(2\pi y_j \xi_j)\Big|\\
&\le 2e^{-\frac{c\kappa(d, N)^2}{4}\sum_{i\in \NN_d\setminus S_z}\sin^2(\pi\xi_i)}
+ 2e^{-\frac{c\kappa(d, N)^2}{4}\sum_{i\in \NN_d\setminus S_z}\cos^2(\pi\xi_i)}.
\end{split}
\end{align}
Assuming momentarily \eqref{eq:62} and exploiting symmetries   we obtain \eqref{eq:59} as follows
\begin{align*}
\frac{1}{|B_N\cap\ZZ^d|}&\Big|\sum_{x\in B_N\cap \ZZ^d\cap E}
\prod_{j=1}^d \cos(2\pi x_j \xi_j)\Big|\\
&\le \frac{2}{|B_N\cap\ZZ^d|}\sum_{z\in Z(B(N))}|Y(N,z)|e^{-\frac{c\kappa(d, N)^2}{4}\sum_{i\in \NN_d\setminus S_z}\sin^2(\pi\xi_i)} \\
&+ \frac{2}{|B_N\cap\ZZ^d|}\sum_{z\in Z(B(N))}|Y(N,z)|e^{-\frac{c\kappa(d, N)^2}{4}\sum_{i\in \NN_d\setminus S_z}\cos^2(\pi\xi_i)}\\
&= \frac{2}{|B_N\cap\ZZ^d|}\sum_{z\in Z(B(N))}|Y(N,z)|\mathbb E\Big[e^{-\frac{c\kappa(d, N)^2}{4}\sum_{i\in \sigma(\NN_d\setminus S_{z})}\sin^2(\pi\xi_i)}\Big] \\
&+ \frac{2}{|B_N\cap\ZZ^d|}\sum_{z\in Z(B(N))}|Y(N,z)|\mathbb E\Big[e^{-\frac{c\kappa(d, N)^2}{4}\sum_{i\in \sigma(\NN_d\setminus S_{z})}\cos^2(\pi\xi_i)}\Big],
\end{align*}
where in the last two lines we used the fact that $Z(B(N))$ is ${\rm Sym}(d)$ invariant and $|Y(N,\sigma\cdot z)|=|Y(N,z)|$ for every $\sigma \in {\rm Sym}(d)$.
Using the fact that $|Y(N,z)|=|z+Y(N,z)|$ and  decomposition \eqref{eq:60} it suffices to show, for every $z\in Z(B(N))$, that
\begin{align}
\label{eq:72}\mathbb E\Big[e^{-\frac{c\kappa(d, N)^2}{4}\sum_{i\in \sigma(\NN_d\setminus S_{z})}\sin^2(\pi\xi_i)}\Big]
&\le 3e^{-\frac{c\kappa(d, N)^2}{100}\sum_{i=1}^d\sin^2(\pi\xi_i)}\\
\label{eq:81}\mathbb E\Big[e^{-\frac{c\kappa(d, N)^2}{4}\sum_{i\in \sigma(\NN_d\setminus S_{z})}\cos^2(\pi\xi_i)}\Big]
&\le 3e^{-\frac{c\kappa(d, N)^2}{100}\sum_{i=1}^d\cos^2(\pi\xi_i)}
\end{align}
     with the constant $c\in(0, 1)$ as in \eqref{eq:70}.
We shall only focus on estimating \eqref{eq:72}, as the bound for \eqref{eq:81} is completely analogous.
For this purpose we will apply Lemma \ref{lem:7}. Observe that for every $z\in Z(B(N))$ we have $|S_z|\le n/4$, and consequently $|\NN_d\setminus S_z|=d-|S_z|\ge d-n/4\ge d(1-\kappa(d, N)^2/4)\ge 99d/100$. Invoking Lemma \ref{lem:7}, with $d_0=0$, $I=\NN_d\setminus S_z$, $\delta_0=49/50$ and $\delta_1=99/100$, we conclude
\begin{align*}
\mathbb E\Big[e^{-\frac{c\kappa(d, N)^2}{4}\sum_{i\in \sigma(\NN_d\setminus S_{z})}\sin^2(\pi\xi_i)}\Big]
\le 3e^{-\frac{c\kappa(d, N)^2}{100}\sum_{i=1}^d\sin^2(\pi\xi_i)}.
\end{align*}

We now  return to the proof of \eqref{eq:62}. Throughout the proof of \eqref{eq:62} we fix $z\in Z(B(N))$ and assume that $Y(N,z)$ is non-empty. We let $I=\NN_d\setminus S_z$ and $m=|I|\ge99d/100$. We shall exploit the properties of the symmetry group restricted to the set $I$, i.e.
${\rm Sym}(I)=\{\sigma \in {\rm Sym}(d)\colon \sigma(y)=y \text{ for every } y\in S_z\}$. Since $\sigma\cdot Y(N,z)=\sigma\cdot Y(N,\sigma^{-1}\cdot z)=Y(N,z)$ for every $\sigma\in {\rm Sym}(I)$ then we have
\begin{align*}
\frac{1}{|Y(N,z)|}\sum_{y\in Y(N,z)}\prod_{j\in\NN_d\setminus S_z} \cos(2\pi y_j \xi_j)
=\frac{1}{|Y(N,z)|}\sum_{y\in Y(N,z)}\frac{1}{m!}\sum_{\sigma \in {\rm Sym}(I)}\prod_{j\in I} \cos(2\pi y_{\sigma^{-1}(j)} \xi_j).
\end{align*}
Our aim is to show that for every $y\in Y(N,z)$ we have 
\begin{align}
\label{eq:61}
\frac{1}{m!}\Big|\sum_{\sigma \in {\rm Sym}(I)}\prod_{j\in I} \cos(2\pi y_{\sigma^{-1}(j)} \xi_j)\Big|\le
2e^{-\frac{c\kappa(d, N)^2}{4}\sum_{i\in \NN_d\setminus S_z}\sin^2(\pi\xi_i)}
     + 2e^{-\frac{c\kappa(d, N)^2}{4}\sum_{i\in \NN_d\setminus S_z}\cos^2(\pi\xi_i)}
     \end{align}
     with the constant $c\in(0, 1)$ as in \eqref{eq:70}.
We fix $y\in Y(N,z)$, we set $l=|\supp y|$, and write
\[
{\rm Sym}(I)=\bigcup_{\substack{J\subseteq I\\ |J|=l}}
\{\sigma \in {\rm Sym}(I)\colon |y_{\sigma^{-1}(j)}|=1 \textrm{ exactly for }j\in J\}.
\]
Using this decomposition we obtain
\begin{align*}
\frac{1}{m!}\sum_{\sigma \in {\rm Sym}(I)}\prod_{j\in I} \cos(2\pi y_{\sigma^{-1}(j)} \xi_j)=
\frac{1}{{m \choose l}}\sum_{\substack{J\subseteq I\\ |J|=l}}\prod_{j\in J} \cos(2\pi  \xi_j).
\end{align*}
We remark that $n/2<l\le n$, since $y+z\in E$, and $m\ge 99d/100\ge n$, since $\kappa(d, N)\le 1/5$. Thus ${m \choose l}$ is well defined.
For $S\subseteq I$ we denote
\begin{align}
\label{eq:63}
a_S(\xi)=\prod_{j\in I\setminus S}\cos^2 (\pi \xi_j) \cdot \prod_{i\in S}\sin^2 (\pi \xi_i),
\end{align}
so that
\begin{align*}
\sum_{S\subseteq I} a_S=\prod_{j \in I}\big(\cos^2(\pi \xi_j) +\sin^2 (\pi \xi_j)\big)=1.
\end{align*}
Taking $\varepsilon(J)\in\{-1,1\}^d$ such that $\varepsilon(J)_j=-1$ precisely for $j \in J$ we may rewrite
\begin{align*}
\prod_{j\in J} \cos(2\pi  \xi_j)&=\prod_{j\in I}\bigg(\frac{1+\cos (2\pi \xi_j)}{2}
+\varepsilon(J)_j\frac{1-\cos (2\pi \xi_j)}{2}\bigg)\\
&=\prod_{j\in I}\big(\cos^2 (\pi \xi_j)+\varepsilon(J)_j\sin^2(\pi \xi_j)\big)\\
&=\sum_{S\subseteq I}a_S(\xi)w_S(\varepsilon(J)),
\end{align*}
where we have defined $w_S\colon \{-1,1\}^d\to \{-1,1\}$ by setting $w_S(\varepsilon)=\prod_{j \in S} \varepsilon_j$. Changing the order of summation we thus have 
\begin{align*}
\frac{1}{m!}\sum_{\sigma \in {\rm Sym}(I)}\prod_{j\in I} \cos(2\pi y_{\sigma^{-1}(j)} \xi_j)&=
\frac{1}{{m \choose l}}\sum_{\substack{J\subseteq I\\ |J|=l}}\prod_{j\in J} \cos(2\pi  \xi_j)\\
&=\sum_{S\subseteq I} a_S(\xi) \frac1{{m \choose l}}\sum_{\substack{J\subseteq I\\ |J|=l}}w_S(\varepsilon(J)).
\end{align*}
     Now, a direct computation shows that
     \begin{align}
     \label{eq:64}
     \begin{split}
\frac1{{m \choose l}}\sum_{\substack{J\subseteq I\\|J|=l}}w_S(\varepsilon(J))
&=\frac{1}{{m \choose l}}\sum_{j=0}^{l}(-1)^j\binom{|S|}{j}\binom{m-|S|}{l-j}\\
&=\Bbbk_{l}^{(m)}(|S|),     
     \end{split}
\end{align}
where $\Bbbk_{l}^{(m)}$ is the Krawtchouk polynomial from \eqref{eq:68}. The first equality in \eqref{eq:64} can be deduced from the disjoint splitting
\[
\{J\subseteq I\colon |J|=l\}=\bigcup_{j=0}^{l}\{ J\subseteq I\colon |J|=l,\ \text{ and }\ |J\cap S|=j\}.
\]
 Indeed, if $|J\cap S|=j$ then $w_S(\varepsilon(J))=(-1)^{j}$ and thus using the above and recalling that $|I|=m$ we write 
 \begin{align*}
 \sum_{\substack{J\subseteq I\\ |J|=l}}w_S(\varepsilon(J))
 =\sum_{j=0}^{l}\sum_{\substack{J\subseteq I\\ |J|=l,\,|J\cap S|=j}}(-1)^{j}
 =\sum_{j=0}^{l} (-1)^{j}\binom{|S|}{j}\binom{m-|S|}{l-j}.
 \end{align*}
 Finally, we obtain
 \begin{align}
 \label{eq:66}
 \frac{1}{m!}\sum_{\sigma \in {\rm Sym}(I)}\prod_{j\in I} \cos(2\pi y_{\sigma^{-1}(j)} \xi_j)
 =\sum_{S\subseteq I} a_S(\xi)\Bbbk_{l}^{(m)}(|S|).
 \end{align}
 By the uniform estimates for Krawtchouk polynomials \eqref{eq:70}, if $0\le l, |S|\le m/2$, we obtain
 \begin{align}
 \label{eq:67}
|\Bbbk_{l}^{(m)}(|S|)|\le e^{-\frac{cl|S|}{m}}. 
 \end{align}
Otherwise, if $|S|>m/2$ and $l\le m/2$, then we use symmetries and \eqref{eq:70} and obtain
\begin{align}
\label{eq:71}
|\Bbbk_{l}^{(m)}(|S|)|=|\Bbbk_{l}^{(m)}(m-|S|)|\le e^{-\frac{cl(m-|S|)}{m}}.
\end{align}
Using \eqref{eq:66}, \eqref{eq:67} and \eqref{eq:71} we obtain
\begin{align*}
\frac{1}{m!}\Big|\sum_{\sigma \in {\rm Sym}(I)}\prod_{j\in I} \cos(2\pi y_{\sigma^{-1}(j)} \xi_j)\Big|&\le
\sum_{\substack{S\subseteq I\\
\emptyset\not=S\not=I}} a_S(\xi)e^{-\frac{cl|S|}{m}}
+\sum_{\substack{S\subseteq I\\
\emptyset\not=S\not=I}} a_S(\xi)e^{-\frac{cl(m-|S|)}{m}}\\
&+e^{-\sum_{i\in I}\sin^2(\pi\xi_i)}
     + e^{-\sum_{i\in I}\cos^2(\pi\xi_i)},
\end{align*}
where $c\in(0, 1)$ is the constant as in \eqref{eq:70}.
     Recalling the definition of $a_S$ from \eqref{eq:63} we have
     \begin{align*}
     \sum_{\substack{S\subseteq I\\
\emptyset\not=S\not=I}} a_S(\xi)e^{-\frac{cl|S|}{m}}&
     +\sum_{\substack{S\subseteq I\\
\emptyset\not=S\not=I}} a_S(\xi)e^{-\frac{cl(m-|S|)}{m}}\\
     &\le \sum_{S\subseteq I}  \prod_{j\in I\setminus S}\cos^2 (\pi \xi_j)
     \cdot \prod_{j\in S}e^{-\frac{cl}{m}}\sin^2 (\pi \xi_j)\\
     &+  \sum_{S\subseteq I}   \prod_{j\in I\setminus S}e^{-\frac{cl}{m}}\cos^2 (\pi \xi_j)
     \cdot \prod_{j\in S}\sin^2 (\pi \xi_j)\\
     &=\prod_{j\in I}\big(\cos^2(\pi \xi_j)+e^{-\frac{cl}{m}}\sin^2(\pi \xi_j)\big)
     +\prod_{j\in I}\big(e^{-\frac{cl}{m}}\cos^2(\pi \xi_j)+\sin^2(\pi \xi_j)\big)\\
     &=\prod_{j\in I}\big(1-(1-e^{-\frac{cl}{m}})\sin^2(\pi \xi_j)\big)
     +\prod_{j\in I}\big(1-(1-e^{-\frac{cl}{m}})\cos^2(\pi \xi_j)\big)\\
     &\le e^{-\frac{c\kappa(d, N)^2}{4}\sum_{i\in I}\sin^2(\pi\xi_i)}
     + e^{-\frac{c\kappa(d, N)^2}{4}\sum_{i\in I}\cos^2(\pi\xi_i)},
     \end{align*}
     where in the last inequality we have applied two simple inequalities $1-x\le e^{-x}$ and $xe^{-x/2}\le 1-e^{-x}$ for all $x\ge0$ and we used the fact that $\frac{1}{2}\kappa(d, N)^2\le \frac{l}{m}\le 2\kappa(d, N)^2$. This completes the proof of \eqref{eq:61}, since  $I=\NN_d\setminus S_z$, and consequently  we complete the proof of Proposition \ref{prop:5}. 
\end{proof}

\subsection{All together}
To prove Proposition \ref{prop:4} we will need the following lemma. 
\begin{lemma}
\label{lem:14}
For every $d,N\in\NN$ and for every $\xi\in\TT^d$ we have
\begin{align}
  \label{eq:75}
  \Big|\mathfrak m_N(\xi)-\frac{1}{|B_N\cap\ZZ^d|}\sum_{x\in B_N\cap\ZZ^d}(-1)^{\sum_{i\in V_{\xi}}x_i}\Big|
  \le 2\kappa(d, N)^{2}\sum_{i=1}^d\cos^2(\pi\xi_i).
\end{align}
\end{lemma}
\begin{proof}
For every $\xi\in\TT^d$, let $\xi'\in\TT^d$ be defined as follows
\begin{align}
  \label{eq:73}
\xi'_i=
  \begin{cases}
  \xi_i & \text{ if } i\not\in V_{\xi},\\
  \xi_i-\frac{1}{2} & \text{ if } i\in V_{\xi} \text{ and } \frac{1}{4}<\xi_i\le\frac{1}{2},\\
  \xi_i+\frac{1}{2} & \text{ if } i\in V_{\xi} \text{ and } -\frac{1}{2}\le\xi_i<-\frac{1}{4},
  \end{cases}
\end{align}
where $V_{\xi}$ is the set defined in \eqref{eq:107}. Hence by \eqref{eq:73} we see that $|\xi_i'|\le1/4$ and
\begin{align}
\label{eq:108}
\sin^2(\pi\xi_i')=\cos^2(\pi\xi_i)\quad \text{for}\quad i\in V_{\xi}.
\end{align}
 Moreover,
\begin{align}
  \label{eq:74}
  \begin{split}
  \mathfrak m_N(\xi)&=\frac{1}{|B_N\cap\ZZ^d|}\sum_{x\in B_N\cap\ZZ^d}\prod_{j=1}^d \cos(2\pi x_j \xi_j)\\
  &=\frac{1}{|B_N\cap\ZZ^d|}\sum_{x\in B_N\cap\ZZ^d}\Big((-1)^{\sum_{i\in V_{\xi}}x_i}\prod_{j=1}^d \cos(2\pi x_j \xi_j')\Big),
  \end{split}
\end{align}
since $\cos(2\pi x_j \xi_j)=\cos(2\pi x_j \xi_j')\cos(\pi x_j)=(-1)^{x_j}\cos(2\pi x_j \xi_j')$.

Arguing in a similar way as in the proof of Proposition \ref{prop:0}  we see, by \eqref{eq:74}, that
\begin{align*}
\Big|\mathfrak m_N(\xi)-\frac{1}{|B_N\cap\ZZ^d|}\sum_{x\in B_N\cap\ZZ^d}(-1)^{\sum_{i\in V_{\xi}}x_i}\Big|&\le\frac{1}{|B_N\cap\ZZ^d|}\sum_{x\in B_N\cap\ZZ^d}\Big(1-\prod_{j=1}^d \cos(2\pi x_j \xi_j')\Big)\\
&\le\frac{2}{|B_N\cap\ZZ^d|}\sum_{x\in B_N\cap\ZZ^d}\sum_{j=1}^d \sin^2(\pi x_j \xi_j')\\
&\le\frac{2}{|B_N\cap\ZZ^d|}\sum_{x\in B_N\cap\ZZ^d}\sum_{j=1}^d x_j^2\sin^2(\pi  \xi_j')\\
&\le\frac{2N^2}{d}\sum_{j=1}^d \sin^2(\pi  \xi_j')\\
&\le 2\kappa(d, N)^{2}\sum_{j=1}^d\cos^2(\pi\xi_j),
\end{align*}
since by \eqref{eq:108} we obtain
\begin{align*}
\sum_{j\in V_{\xi}}\sin^2(\pi  \xi_j')=\sum_{j\in V_{\xi}}\cos^2(\pi  \xi_j),\quad \text{ and }\quad
\sum_{j\in \NN_d\setminus V_{\xi}}\sin^2(\pi  \xi_j')\le\sum_{j\in \NN_d\setminus V_{\xi}}\cos^2(\pi  \xi_j').
\end{align*}
  The proof of \eqref{eq:75} is completed.
\end{proof}

\begin{proof}[Proof of Proposition \ref{prop:4}]
Firstly, we assume that $|V_{\xi}|\le d/2$, then $|\NN_d\setminus V_{\xi}|=d-|V_{\xi}|\ge d/2$ and \eqref{eq:98} implies that
\begin{align*}
	|\mathfrak m_N(\xi)|\le 16e^{-\frac{c\kappa(d, N)^2}{400}\sum_{i=1}^d\sin^2(\pi\xi_i)},
\end{align*}
since $\sum_{i=1}^d\cos^2(\pi\xi_i)\ge d/4\ge 1/4\sum_{i=1}^d\sin^2(\pi\xi_i)$. Thus
\begin{align}
\label{eq:110}
\Big|\mathfrak m_N(\xi)-e^{-\kappa(d, N)^2\sum_{i=1}^d\sin^2(\pi\xi_i)}\Big|\le 17e^{-\frac{c\kappa(d, N)^2}{400}\sum_{i=1}^d\sin^2(\pi\xi_i)}.
\end{align}
On the other hand, using \eqref{eq:mN-1} from the proof of Proposition \ref{prop:0} we obtain 
\begin{align}
\label{eq:111}
\Big|\mathfrak m_N(\xi)-e^{-\kappa(d, N)^2\sum_{i=1}^d\sin^2(\pi\xi_i)}\Big|\le 3\kappa(d, N)^2\sum_{i=1}^d\sin^2(\pi\xi_i).
\end{align}
We now see that \eqref{eq:110} and \eqref{eq:111} imply \eqref{eq:94}.

Secondly, we assume that $|V_{\xi}|\ge d/2$, then \eqref{eq:98} implies that
\begin{align*}
	|\mathfrak m_N(\xi)|\le 16e^{-\frac{c\kappa(d, N)^2}{400}\sum_{i=1}^d\cos^2(\pi\xi_i)},
\end{align*}
since $\sum_{i=1}^d\sin^2(\pi\xi_i)\ge d/4\ge 1/4\sum_{i=1}^d\cos^2(\pi\xi_i)$. Thus
\begin{align}
\label{eq:112}
\Big|\mathfrak m_N(\xi)-\frac{1}{|B_N\cap\ZZ^d|}\Big(\sum_{x\in B_N\cap\ZZ^d}(-1)^{\sum_{i=1}^dx_i}\Big)
e^{-\kappa(d, N)^2\sum_{i=1}^d\cos^2(\pi\xi_i)}\Big| \le 17e^{-\frac{c\kappa(d, N)^2}{400}\sum_{i=1}^d\cos^2(\pi\xi_i)}.
\end{align}
On the other hand, by Lemma \ref{lem:14} we obtain
\begin{align}
\label{eq:113}
  \Big|\mathfrak m_N(\xi)-\frac{1}{|B_N\cap\ZZ^d|}\Big(\sum_{x\in B_N\cap\ZZ^d}(-1)^{\sum_{i\in V_{\xi}}x_i}\Big)e^{-\kappa(d, N)^2\sum_{i=1}^d\cos^2(\pi\xi_i)}\Big|
  \le 3\kappa(d, N)^{2}\sum_{i=1}^d\cos^2(\pi\xi_i).
\end{align}
Moreover, arguing in a similar way as in the proof of Proposition \ref{prop:0}  we see that
\begin{align}
\label{eq:114}
\begin{split}
\bigg|\frac{1}{|B_N\cap\ZZ^d|}\sum_{x\in B_N\cap\ZZ^d}(-1)^{\sum_{i\in V_{\xi}}x_i}
&-\frac{1}{|B_N\cap\ZZ^d|}\sum_{x\in B_N\cap\ZZ^d}(-1)^{\sum_{i=1}^dx_i}\bigg|\\
&\le\frac{1}{|B_N\cap\ZZ^d|}\sum_{\substack{x\in B_N\cap\ZZ^d\\
  (\NN_d\setminus V_{\xi})\cap {\rm supp}\:x\not=\emptyset}}\big|(-1)^{\sum_{i\in\NN_d\setminus V_{\xi}}x_i}-1\big|\\
  &\le\frac{2}{|B_N\cap\ZZ^d|}\sum_{x\in B_N\cap\ZZ^d}\sum_{i\in \NN_d\setminus V_{\xi}}x_i^2\\
  &\le 2\kappa(d, N)^2(d-|V_{\xi}|)\\
  &\le 4\kappa(d, N)^2\sum_{i=1}^d\cos^2(\pi\xi_i),
\end{split}
  \end{align}
since $1/2\le \cos^2(\pi\xi_i)$ for any $i\in \NN_d\setminus V_{\xi}$ and consequently $d-|V_{\xi}|\le 2\sum_{i\in \NN_d\setminus V_{\xi}}\cos^2(\pi\xi_i)$. Combining \eqref{eq:113} and \eqref{eq:114} we obtain 
\begin{align}
\label{eq:115}
 \Big|\mathfrak m_N(\xi)-\frac{1}{|B_N\cap\ZZ^d|}\Big(\sum_{x\in B_N\cap\ZZ^d}(-1)^{\sum_{i=1}^dx_i}\Big)
e^{-\kappa(d, N)^2\sum_{i=1}^d\cos^2(\pi\xi_i)}\Big| \le 7\kappa(d, N)^2\sum_{i=1}^d\cos^2(\pi\xi_i).
\end{align}
We now see that \eqref{eq:112} and \eqref{eq:115} imply \eqref{eq:95}. This completes the proof.
\end{proof}

\end{document}